\documentclass[11pt,reqno]{amsart}
\usepackage{amssymb}
\usepackage{amsfonts}
\usepackage{amsmath}
\usepackage{amssymb, amsmath}
\usepackage{graphicx}
\usepackage{color}

\newcommand{\newcom}{\newcommand}

\newcom{\al}{\alpha}
\newcom{\be}{\beta}
\newcom{\eps}{\epsilon}
\newcom{\veps}{\varepsilon}
\newcom{\ga}{\gamma}
\newcom{\Ga}{\Gamma}
\newcom{\ka}{\kappa}
\newcom{\Lam}{\Lambda}
\newcom{\lam}{\lambda}
\newcom{\Om}{\Omega}
\newcom{\om}{\omega}
\newcom{\Si}{\Sigma}
\newcom{\si}{\sigma}
\newcom{\tht}{\theta}
\newcom{\dtri}{\nabla}
\newcom{\tri}{\triangle}
\newcom{\oo}{\infty}
\newcom{\vphi}{\varphi}
\newcom{\cB}{{\mathcal B}}
\newcom{\cC}{{\mathcal C}}
\newcom{\cD}{{\mathcal D}}
\newcom{\cF}{{\mathcal F}}
\newcom{\cL}{{\mathcal L}}
\newcom{\cM}{{\mathcal M}}
\newcom{\cP}{{\mathcal P}}
\newcom{\cS}{{\mathcal S}}
\newcom{\cQ}{{\mathcal Q}}
\newcom{\cT}{{\mathcal T}}
\newcom{\cY}{{\mathcal Y}}
\newcom{\cZ}{{\mathcal Z}}
\newcom{\R}{\mathbb R}
\newcom{\T}{\mathbb T}
\newcom{\N}{\mathbb N}
\newcom{\Z}{\mathbb Z}
\newcom{\C}{\mathbb C}
\newcom{\E}{\mathbb E}

\newcommand{\bu}{u}
\def\R{\mathbb{R}}
\newcommand{\D}{\displaystyle}

\newcommand{\FP}{\mathbf{P}}

\newcom{\f}{\frac}
\newcom{\di}{\displaystyle\int}
\newcom{\ds}{\displaystyle\sum}
\newcom{\dl}{\displaystyle\lim}
\newcom{\ov}{\overline}
\newcom{\sset}{\subset}
\newcom{\wt}{\widetilde}
\newcom{\pa}{\partial}
\newcom{\p}{\partial}
\newcom\na{\nabla}
\newcom{\suml}{\sum\limits}
\newcom{\supl}{\sup\limits}
\newcom{\intl}{\int\limits}
\newcom{\infl}{\inf\limits}
\newcom{\disp}{\displaystyle}
\newcom{\non}{\nonumber}
\newcom{\no}{\noindent}
\newcom{\QED}{$\square$}

\newtheorem{athm}{\bf \t}[section]
\newenvironment{thm} [1] {\def\t{#1}\begin{athm} \bf \rm} {\end{athm}}
\newcom{\bthm}{\begin{thm}}
\newcom{\ethm}{\end{thm}}

\newtheorem{theorem}{Theorem}[section]
\newtheorem{lemma}{Lemma}[section]
\newtheorem{remark}{Remark}[section]
\newtheorem{corollary}{Corollary}[section]

\newtheorem{proposition}{Proposition}[section]

\newcom{\beq}{\begin{equation}}
\newcom{\eeq}{\end{equation}}
\newcom{\ben}{\begin{eqnarray}}
\newcom{\een}{\end{eqnarray}}
\newcom{\beno}{\begin{eqnarray*}}
\newcom{\eeno}{\end{eqnarray*}}
\newcom{\bali}{\begin{aligned}}
\newcom{\eali}{\end{aligned}}

\numberwithin{equation}{section}

\begin{document}

\title[A quasi-linear hyperbolic-parabolic system for vasculogenesis]
{Asymptotic stability of  diffusion waves of a quasi-linear hyperbolic-parabolic model for vasculogenesis}

\author{Qingqing Liu}
\address{(QQL)
School of Mathematics, South China University of Technology,
Guangzhou, 510641, P. R. China} \email{maqqliu@scut.edu.cn}

\author{Hongyun  Peng}
\address{(HYP)
School of Applied Mathematics, Guangdong University of Technology, Guangzhou, 510006, China} \email{penghy010@163.com}

\author{Zhi-An Wang}
\address{(ZAW)
Department of Applied Mathematics, Hong Kong Polytechnic University,
Hung Hom, Kowloon, Hong Kong} \email{mawza@polyu.edu.hk}


\date{\today}

\begin{abstract}
{In this paper, we derive the large-time profile of solutions to the Cauchy problem of a hyperbolic-parabolic system modeling the vasculogenesis in $\R^3$.  When the initial data are prescribed in the vicinity of a constant ground state, by constructing a time-frequency Lyapunov functional and employing the Fourier energy method and spectral analysis,   we show that solution of the Cauchy problem tend time-asymptotically to  linear diffusion waves around the constant ground state with algebraic decaying rates under certain conditions on the density-dependent pressure function.}
\end{abstract}

\maketitle


\section{Introduction}

This paper is concerned with the following quasi-linear hyperbolic-parabolic system describing vasculogenesis 
\begin{eqnarray}\label{HPV1}
&&\left\{\begin{aligned}
& \pa_{t}\rho+\nabla\cdot(\rho u)=0,\\
&\pa_t(\rho u)+\nabla\cdot(\rho u\otimes u)+\nabla P(\rho)=\mu \rho\nabla \phi-\alpha \rho u,\\
&\pa_t \phi=D\Delta \phi+a \rho-b \phi,
\end{aligned}\right.
\end{eqnarray}
where $(x,t)\in \R^3\times (0, \infty)$. The model \eqref{HPV1} was proposed in \cite{Gamba} to reproduce key features of experiments of {\it in vitro} formation of blood vessels showing that cells randomly spreading on a gel matrix autonomously organize to a connected vascular network (more extensive modeling details can be found in \cite{ambrosi2005review}), where the unknowns $\rho=\rho(t,x)\geq 0 $ and $
u=u(t,x)\in \mathbb{R}^{3}$ denote the density and velocity of endothelial  cells, respectively, and  $\phi=\phi(t,x)\geq 0$ denotes the concentration of the chemoattractant secreted by the endothelial  cells. The convection term $\nabla \cdot (\rho \bu \otimes \bu)$ models the cell movement persistence (inertial effect), $P(\rho)$ is the cell-density dependent pressure function accounting for the fact that closely packed cells resist to compression due to the impenetrability of cellular matter, the parameter $\mu>0$ measures the intensity of cell response to the chemoattractant concentration gradient and $-\alpha \rho \bu$ corresponds to a damping (friction) force with coefficient $\alpha>0$ as a result of the interaction between cells and the underlying substratum; $D>0$ is the diffusivity of the chemoattractant, the positive constants $a$ and $b$ denote the secretion and death rates of the chemoattractant, respectively.

At the fist sight the hyperbolic-parabolic system \eqref{HPV1} is analogous to the Euler-Poisson system with damping:
\begin{eqnarray} \label{ep}
\left\{\begin{array}{l}
\partial_t \rho+\nabla \cdot (\rho \bu)=0, \\[1mm]
\partial_t (\rho \bu)+\nabla \cdot (\rho \bu \otimes \bu)+\nabla P(\rho)=\mu \rho \nabla \phi -\alpha \rho \bu,\\[1mm]
-\Delta \phi=a \rho-\mathcal{K}(x)
 \end{array} \right.
\end{eqnarray}
where $\rho, \bu$ and $\phi$ denote the density, velocity and potential of the flows, respectively; $P(\rho)$ is the density-dependent pressure function, and $\mathcal{K}(x)\geq 0$ is the background state (doping profile). The damped Euler-Poisson system appears in numerous important applications including the propagation of electrons in semiconductor devices (cf. \cite{markowich2012semiconductor}) and the transport of ions in plasma physics (cf. \cite{choudhuri1998physics}) when $\mu<0$ and $\alpha\geq0$, as well as the collapse of gaseous stars due to self-gravitation \cite{chandrasekhar1957introduction} when $\mu>0$ and $\alpha=0$. Without potential (i.e. $\phi=0$), the system \eqref{ep} reduces to the well-known Euler equations. The system \eqref{HPV1} has several essential differences from \eqref{ep}. First, the parabolic governing equation for $\phi$ in \eqref{HPV1} is harder to deal with than the elliptic governing equation in \eqref{ep}. For example, in one dimension, the sign of $\phi_x$ can be directly determined from $L^1$-norms of $\rho$ and $\mathcal{K}(x)$ from the elliptic equation in \eqref{ep}, but it is elusive from the parabolic equation in \eqref{HPV1}.  Second, the equation for $\phi$ in \eqref{HPV1} has a decay term $-b\phi$ different from a given background state $\mathcal{K}(x) \geq 0$ in \eqref{ep}. The background state $\mathcal{K}(x)$ often directly determines the large-time profile of $\rho$, but the large time behavior of $\rho$ in \eqref{HPV1} is obscure. These structural differences bring substantial differences to the model dynamics and difficulties in the analysis. Many of the mathematical methods developed in the literature for the Euler-Poisson system \eqref{ep} are inapplicable to \eqref{HPV1}, while existing results available to \eqref{HPV1} are rather limited. The goal of this paper is to explore the possible asymptotic profile of solutions to \eqref{HPV1} without vacuum (i.e. $\rho(x,t)>0$ for all $(x,t) \in  \R^3 \times [0,\infty)$), which is closely related to its steady states. It can be check that the system \eqref{HPV1} possesses the following energy functional (cf. \cite{berthelin2016stationary, chavanis2007kinetic})
$$F[\rho, u, \phi]=\frac{1}{2\mu }\int_{\R^3}\rho u^2 d x+\frac{1}{\mu }\int_{\R^3}G(\rho) d x+\frac{1}{2a}\int_{\R^3} (|\nabla \phi|^2+b \phi^2) dx-\int_{\R^3} \rho \phi d x$$
where $ G''(\rho)=P'(\rho)/\rho$, which satisfies
$$\frac{d}{dt} F[\rho, u, \phi]+\frac{\alpha}{\mu}\int_{\R^3}\rho u^2 d x+\frac{1}{a}\int_{\R^3}|\phi_t|^2 dx=0.$$
Thus the stationary solution satisfies $\frac{d}{dt} F[\rho, u, \phi]=0$ which gives rise to $\rho u=0$ and $\phi_t=0$ in $\mathbb{R}^3$. Since we are interested in non-constant profile for $\rho$, $u\equiv 0$ is the only (physical) stationary profile for the velocity $u$. In the literature, the following initial data for \eqref{HPV1} are considered
\begin{eqnarray}\label{1.2}
[\rho,u,\phi]|_{t=0}=[\rho_0,u_{0},\phi_{0}](x)\rightarrow [\bar{\rho}, 0,\bar{\phi}]\quad \mathrm{as} \quad |x|\rightarrow \infty
\end{eqnarray}
for some  constants  $\bar{\rho}>0$ and $\bar{\phi}>0$. When the initial value $[\rho_0, u_0, \phi_0] \in H^s(\R^d) (s>d/2+1)$  is a small perturbation of the constant ground state (i.e. equilibrium) $[\bar{\rho}, 0, \bar{\phi}]$ with $\bar{\rho}>0$ sufficiently small, it was shown in \cite{Russo12, Russo13} that the system \eqref{HPV1} with \eqref{1.2} admits global strong solutions without vacuum  converging to $[\bar{\rho}, 0, \bar{\phi}]$ with an algebraic rate $(1+t)^{-\frac{3}{4}}$ as $t \to \infty$. As $\alpha\rightarrow \infty$ (strong damping), it was formally derived in \cite{chavanis2007kinetic} by the asymptotic analysis and subsequently justified in \cite{FD} that the solution of \eqref{HPV1} converges to that of a parabolic-elliptic Keller-Segel type chemotaxis system.
By adding a viscous term $\Delta u $ to the second equation of \eqref{HPV1}, the linear stability of the constant ground state $[\bar{\rho}, 0, \bar{\phi}]$ was obtained in \cite{kowalczyk2004stability} under the condition
\begin{equation}\label{pressure}
bP'(\bar{\rho})- a \mu \bar{\rho}>0.
\end{equation}
A typical form of $P$ fulfilling \eqref{pressure} is $ P(\rho) = \frac{K}{2}\rho ^{2}$ with $ K > \frac{a \mu}{b}$.
The stationary solutions of \eqref{HPV1} with vacuum (bump solutions) in a bounded interval with zero-flux boundary condition were constructed in \cite{berthelin2016stationary, carrillo2019phase}. The model \eqref{HPV1} with $P(\rho)=\rho$ and periodic boundary conditions in one dimension was numerically explored in \cite{filbet2005approximation}. Recently the stability of transition layer solutions of \eqref{HPV1} on $\R_+=[0,\infty)$  was established in \cite{HPWZ-JLMS}.

In this paper, we shall fully exploit the special structures of \eqref{HPV1} and find the refined large-time profile of solutions of \eqref{HPV1} with \eqref{1.2}.  It is well known that the frictional damping may generate nonlinear diffusive waves for the hyperbolic equations like the $p$-system (cf. \cite{HL,JZ1,JZ2,Mei1,Mei2,N1,N2,NY,ZJ} without vacuum and \cite{DP,HMP, MM,Pa} with vacuum), bipolar hydrodynamical model of semiconductors \cite{GHL}, bipolar Euler-Maxwell equation \cite{DLZ1}, Timoshenko system \cite{IK}, the radiating gas model \cite{LK} and so on.  Inspired by these works, we expect that when the initial data $[\rho_0, u_0, \phi_0]$ are close to the constant ground state $[\bar{\rho},0,\bar{\phi}]$, the asymptotic profile of solutions to \eqref{HPV1} with \eqref{1.2}  under the condition \eqref{pressure} can be approximated by diffusion waves. To see how this is possible, we first observe that due to the external frictional force,  the inertial terms in the momentum equation of \eqref{HPV1} decay to zero faster than other terms so that the pressure gradient force is balanced by the frictional force plus the potential force. Since we are concerned with the large-time dynamics of solutions near the constant equilibrium, we may speculate the time-asymptotic dynamics described by the third equation of \eqref{HPV1} is mainly determined by the equation $a\rho -b \phi=0$. Therefore if we define the mass flux $m=\rho u$, we anticipate that the solution
$(\rho, m, \phi)$ of system \eqref{HPV1} as $t \to \infty$ may behave as the solution $(\tilde{\rho}, \tilde{m}, \tilde{\phi})$ to the following decoupled equations
\begin{equation}\label{Q1eq2}
\begin{cases}
\tilde{\rho}_t-\Delta  q(\tilde{\rho})=0,  &\ \ \text{(Porous medium equation)}\\
\tilde{u}=- \frac{1}{\tilde{\rho}}\nabla q(\tilde{\rho}), &\ \ \text{(Darcy law)}\\
\tilde{\phi}=\frac{a}{b}\tilde{\rho}
\end{cases}
\end{equation}
where $q(\tilde{\rho})=\frac{1}{\alpha}(P(\tilde{\rho})-\frac{a\mu}{2b}\tilde{\rho}^2)$. We remark that the first equation of \eqref{Q1eq2} is a porous medium equation and the second equation of \eqref{Q1eq2} indeed complies with the Darcy law by noticing that the $\tilde{\rho} \tilde{u}$ defines the mass flux. Since we shall focus on the asymptotic profile of solutions near the constant ground state $[\bar{\rho},0,\bar{\phi}]$, we are motivated to linearize the above equations at $[\bar{\rho},0,\bar{\phi}]$.
Notice that the linearization of the first equation of \eqref{Q1eq2} at $\bar{\rho}$ yields a linear heat equation
$$\partial_{t}\tilde{\rho}-\sigma\Delta \tilde{\rho}=0,\ \ \sigma=\frac{bP(\bar{\rho})-a\mu \bar{\rho}}{b \alpha}$$
which has a unique solution in the form of diffusion wave profile: $\tilde{\rho}(t,x)=G(t,\cdot)*(\rho_0-\bar{\rho})$, where
$$G(t,x)=(4\pi \sigma t)^{-3/2}e^{-\frac{|x|^2}{4\sigma t}}$$
 is the heat kernel.  Therefore we derive that a possible time-asymptotic profile of  solutions to \eqref{HPV1}, \eqref{1.2} may be explicitly given by
\begin{equation}\label{def.pp3}
\begin{aligned}
\tilde{\rho}(t,x)=G(t,\cdot)*(\rho_0-\bar{\rho}),\
\tilde{u}(t,x)= -\frac{bP(\bar{\rho})-a\mu \bar{\rho}}{b \bar{\rho}\alpha}\nabla\tilde{\rho},\
\tilde{\phi}(t,x)=\frac{a}{b}\tilde{\rho}.
\end{aligned}
\end{equation}

The main objective of this paper is to justify that the time-asymptotic profile of solutions to \eqref{HPV1} with \eqref{1.2} is given by $[\tilde{\rho}+\bar{\rho}, \tilde{u}, \tilde{\phi}+\frac{a}{b}\bar{\rho}]$ under the condition \eqref{pressure} if the initial value is in the vicinity of the constant ground state $[\bar{\rho}, 0, \frac{a}{b}\bar{\rho}]$.  The specific results are given in the following theorem.

\begin{theorem}\label{thm.GE}
For any $\bar{\rho}>0$ with $\bar{\phi}=\frac{a}{b}\bar{\rho}$, if the condition \eqref{pressure} holds, 
 then there exists a constant $\eps>0$  such that if
\begin{equation*}
\label{thm.as1}
 \|[\rho_0-\bar{\rho},u_{0},\phi_{0}-\bar{\phi}]\|_{H^{N}(\mathbb{R}^{3})} +\|\nabla \phi_0]\|_{H^{N}(\mathbb{R}^{3})}<\eps,
\end{equation*}
with $N\geq 3$, then the Cauchy problem \eqref{HPV1}, \eqref{1.2} admits a unique global solution such that
\begin{eqnarray*}
\begin{aligned}
&[\rho(t,x)-\bar{\rho},u(t,x),\phi(t,x)-\bar{\phi},\nabla \phi]\in
C([0,\infty);H^{N}(\mathbb{R}^{3}))
\end{aligned}
\end{eqnarray*}
satisfying
\begin{equation}\label{gl1}
\|[\rho(t)-\bar{\rho},u(t),\phi(t)-\bar{\phi},\nabla \phi(t)]\|^{2}_{H^N(\mathbb{R}^{3})}
\leq C_{0}
\|[\rho_0-\bar{\rho},u_{0},\phi_{0}-\bar{\phi},\nabla \phi_{0}]\|^{2}_{H^N(\mathbb{R}^{3})}
\end{equation}
for some constant $C_0>0$ independent of $t$. Moreover, there are constants $\delta_{1}>0$ and $ C_{1}>0$ such that if
\begin{eqnarray*}
\|[\rho_0-\bar{\rho},u_{0},\phi_{0}-\bar{\phi},\nabla \phi_{0}]\|_{H^4(\mathbb{R}^{3})}+\|[\rho_0-\bar{\rho},u_{0},\phi_{0}-\bar{\phi}]\|_{L^{1}(\mathbb{R}^{3})}\leq
\delta_{1},
\end{eqnarray*}
then the solution $[\rho(t,x),u(t,x),\phi(t,x)] $
satisfies for all $t\geq 0$ that,
\begin{eqnarray}\label{LQ}
\|[\rho-\bar{\rho},\phi-\bar{\phi}]\|_{L^q(\mathbb{R}^{3})} \leq C_{1}
(1+t)^{-\frac{3}{2}+\frac{3}{2q}},\quad \ \ \|u\|_{L^q(\mathbb{R}^{3})}\leq
C_{1}(1+t)^{-2+\frac{3}{2q}},
\end{eqnarray}
and
\begin{eqnarray} \label{nu.decay3}
\|[\rho-\bar{\rho}-\tilde{\rho},\phi-\bar{\phi}-\tilde{\phi}]\|_{L^q(\mathbb{R}^{3})} \leq C_{1}
(1+t)^{-2+\frac{3}{2q}},\quad \ \ \|u-\tilde{u}\|_{L^q(\mathbb{R}^{3})}\leq
C_{1}(1+t)^{-\frac{5}{2}+\frac{3}{2q}},
\end{eqnarray}
for any  $2\leq q\leq +\infty$.
\end{theorem}

\begin{remark}\em{
It is shown in Theorem \ref{thm.GE} that the solution of \eqref{HPV1} with \eqref{1.2} converges to the constant ground state $[\bar{\rho}, 0, \bar{\phi}]$ in $L^2(\R^3)$ with the following convergence rate
\begin{equation*}
\|[\rho-\bar{\rho},\phi-\bar{\phi}\|_{L^2(\R^3)}\leq C (1+t)^{-\frac{3}{4}}, \ \ \ \ \|u\|_{L^2(\R^3)}\leq C (1+t)^{-\frac{5}{4}},
\end{equation*}
which comply with the results shown in \cite{Russo12, Russo13} under the condition that $\bar{\rho}>0$ is small. The main contributions of this paper consist of three parts. First we discover the diffusion wave profile $[\tilde{\rho}, \tilde{u}, \tilde{\phi}]$ given by \eqref{def.pp3} under the assumption \eqref{pressure} without assuming $\bar{\rho}$ is small. Second we show that the solution $[\rho, u, \phi]$ of the Cauchy problem \eqref{HPV1} with \eqref{1.2} can converge to the asymptotic profile $[\tilde{\rho}+\bar{\rho}, \tilde{u}, \tilde{\phi}+\bar{\phi}]$ in $L^q$-norm for any $2\leq q \leq \infty$. The works \cite{Russo12, Russo13} only show the convergence of solutions to the constant ground state $[\bar{\rho}, 0, \bar{\phi}]$ in $L^2$-norm.  Third, we prove that the convergence rate to the asymptotic profile $[\tilde{\rho}+\bar{\rho}, \tilde{u}, \tilde{\phi}+\bar{\phi}]$ is faster than  the one to the constant ground state $[\bar{\rho}, 0, \bar{\phi}]$, as \eqref{nu.decay3} compared to \eqref{LQ}.
}
\end{remark}

\begin{remark}\em{
The proof of Theorem \ref{thm.GE} critically depends on the structure of the $\phi$-equation of \eqref{HPV1} with $b\ne 0$. Hence the results of Theorem \ref{thm.GE} can not be carried over to the damped Euler-Poisson system \eqref{ep} and hence we identify an essential difference between damped Euler-Poisson system \eqref{ep} and the hyperbolic-parabolic system \eqref{HPV1}.
}
\end{remark}

To prove our results, we first employ the Fourier energy method to construct  a time-frequency Lyapunov functional and  perform the delicate spectral analysis on the linearized problem to find its refined decay structure.  Then we decompose  solutions of the Cauchy problem \eqref{HPV1} into two parts: linearized solutions around the asymptotic profile $[\tilde{\rho}, \tilde{u}, \tilde{\phi}]$  and perturbed solutions around the linearized solutions, and finally manage to obtain the faster decay rates by using the refined decay properties derived in the forgoing spectral analysis. As far as we know these methods have not been used for the hyperbolic-parabolic system \eqref{HPV1} in the literature.

The rest of the paper is organized as follows. In Section \ref{sec2}, we  study some decay properties of the linearized system around the constant ground state by the Fourier analysis. In Section \ref{sec4}, we conduct  spectral analysis for the linearized system and derive the asymptotic decay rates of linearized solutions towards the linear wave profiles. In Section \ref{sec5}, we first derive some {\it a priori} estimates to obtain the global existence of solutions by the method of energy estimates and  find the time asymptotic decay rate of solutions towards the constant ground state, and finally prove our main results stated in Theorem \ref{thm.GE}.

\section{Decay property of the linearized system}\label{sec2}

In this section, we study the time-decay property of solutions to the linearized system based on the  Fourier energy method. The main purpose here is to exploit the linear dissipative structure to see how the condition \eqref{pressure} plays an important role in the energy estimates.  Notice that the key estimate \eqref{s4.j} in this section will be also used to explore the time-decay property of solutions with the high-frequency in the subsequent sections.
Before we proceed, we introduce some notations frequently used in the paper. \\

\noindent{\bf Notations.} Throughout this paper,
$C$ denotes a generic positive (generally large) constant and $\lambda$ denotes some positive (generally small) constant, where both $C$ and
$ \lambda$ may take different values in different places. For two
quantities $a$ and $b$, $a\sim b$ means $\lambda a \leq  b \leq
\frac{1}{\lambda} a $ for a generic constant $0<\lambda<1$. For any
integer $m\geq 0$, we use $H^{m}$, $\dot{H}^{m}$ to denote the usual
Sobolev space $H^{m}(\mathbb{R}^{3})$. For simplicity, the norm of $ H^{m}$ is denoted by
$\|\cdot\|_{m} $ with $\|\cdot \|=\|\cdot\|_{0}$. We use $
\langle\cdot, \cdot \rangle$ to denote the inner product of the
Hilbert space $ L^{2}(\mathbb{R}^{3})$, i.e.
\begin{eqnarray*}
\langle f,g \rangle=\int_{\mathbb{R}^{3}} f(x)g(x)dx, \forall  f, g  \in L^2(\mathbb{R}^{3}).
\end{eqnarray*}
For a multi-index $l=(l_{1},l_{2},l_{3})$, we denote $\partial^{l}=\partial^{l_{1}}_{x_{1}}\partial^{l_{2}}_{x_{2}}\partial^{l_{3}}_{x_{3}}$ and
the length of $l$ is $|l|=l_{1}+l_{2}+l_{3}$. For simplicity, we denote
\begin{equation*}
\|[A,B]\|_{X}=\|A\|_{X}+\|B\|_{X},
\end{equation*}
for some Sobolev space $X$.

\subsection{Decay structure of the linearized system}
In this section, we  use $[\rho_1,u,\phi_1]$ to denote
the solution of the linearized system of \eqref{HPV1} around the constant ground state $[\bar{\rho}, 0, \bar{\phi}]$
\begin{equation}\label{HPVL}
\left\{\begin{aligned}
& \pa_{t}\rho_1+\bar{\rho}\nabla\cdot u=0,\\
&\pa_t u+\frac{P'(\bar{\rho})}{\bar{\rho}}\nabla \rho_1+\alpha u -\mu\nabla\phi_1=0,\\
&\pa_t \phi_1-D\Delta \phi_1-a\rho_1+b\phi_1=0,
\end{aligned}\right.
\end{equation}
with initial data
\begin{eqnarray}\label{NLI}
[\rho_1,u,\phi_1]|_{t=0}=[\rho_{1,0},u_{0},\phi_{1,0}]=[\rho_0-\bar{\rho},u_{0},\phi_{0}-\bar{\phi}].
\end{eqnarray}

The goal of this section is to apply the Fourier energy method  to the Cauchy problem \eqref{HPVL}-\eqref{NLI} to find a
time-frequency Lyapunov functional proportional to $|\hat{U}(t,k)|^{2}+|k|^2|\hat{\phi}|^{2}$ with a dissipative structure.
In the rest of this section, for the simplicity of notations, we shall use $[\rho, u, \phi]$ to denote the solution  $[\rho_1, u, \phi_1]$ of $\eqref{HPVL}$-$\eqref{NLI}$, correspondingly, $ [\rho_{0}, u_{0}, \phi_{0}]$ to denote $[\rho_{1,0}, u_{0}, \phi_{1,0}]$ unless other stated.
Let us present the main result of this section as follows.

\begin{theorem}\label{lyapunov} Let $U(t,x)=[\rho, u, \phi]$ be a well-defined solution to the system
	$\eqref{HPVL}$-$\eqref{NLI}$. Then there is a time-frequency Lyapunov
	functional $\mathcal{E}(\hat{U}(t,k))$ with
	\begin{eqnarray}\label{equver}
	\mathcal {E}(\hat{U}(t,k))\sim
	|[\hat{\rho},\hat{u},\hat{\phi}]|^{2}+|k|^2|\hat{\phi}|^{2}
	\end{eqnarray}
	such that for any $ t>0$ and $k\in\mathbb{R}^{3}$ the Lyapunov inequality
	\begin{eqnarray}\label{Lyainequlity1}
	\frac{d}{dt}\mathcal {E}(\hat{U}(t,k))+\dfrac{\lambda|k|^{2}}{1+|k|^{2}}\mathcal
	{E}(\hat{U}(t,k))\leq 0
	\end{eqnarray}
	holds for some constant $\lambda>0$.
\end{theorem}

\begin{proof}
	For an integrable function $f:\mathbb{R}^{3} \rightarrow
	\mathbb{R}$, its Fourier transform is defined by
	\begin{eqnarray*}
		\hat{f}(k)=\int_{\mathbb{R}^{3}}{\exp({-ix\cdot k})}f(x)dx,\ \ \
		x\cdot k:=\sum_{j=1}^{3}x_{j}k_{j},\ \ k\in \mathbb{R}^{3},
	\end{eqnarray*}
	where $i=\sqrt{-1}\in\mathbb{C}$ is  the imaginary unit. For two
	complex numbers or vectors $a$ and $b$, $(a|b)$ denotes the dot
	product of $a$ with the complex conjugate of $b$. Taking the Fourier
	transform in $x$ for $\eqref{HPVL}$, we find
	$\hat{U}=[\hat{\rho},\hat{u},\hat{\phi}]$
	satisfies
	\begin{equation}\label{LinearF}
	\left\{\begin{aligned}
	& \pa_{t}\hat{\rho}+\bar{\rho}ik\cdot \hat{u}=0,\\
	&\pa_t \hat{u}+\frac{P'(\bar{\rho})}{\bar{\rho}}ik\hat{\rho}+\alpha \hat{u} -\mu ik\hat{\phi}=0,\\
	&\pa_t \hat{\phi}+D|k|^2\hat{\phi}-a\hat{\rho}+b\hat{\phi}=0.
	\end{aligned}\right.
	\end{equation}
	
	First of all, it is straightforward to obtain from the first two
	equations of $\eqref{LinearF}$ that
	\begin{eqnarray}\label{LinearF1}
	\begin{aligned}
	& \frac{1}{2}\frac{d}{dt} \left(\frac{P'(\bar{\rho})}{\bar{\rho}}|\hat{\rho}|^2+\bar{\rho}|\hat{u}|^2\right)+\alpha \bar{\rho}|\hat{u}|^2\\
	=&\mathrm{Re}(\mu ik\hat{\phi}|\bar{\rho}\hat{u}) =-\mathrm{Re}(\mu\hat{\phi}|\bar{\rho}ik\cdot\hat{u})\\
	=&\mathrm{Re}(\mu\hat{\phi}| \hat{\rho}_{t}) =\frac{d}{dt}\mathrm{Re}(\mu\hat{\phi}| \hat{\rho})-\mathrm{Re}(\mu\hat{\phi_{t}}| \hat{\rho}).
	\end{aligned}
	\end{eqnarray}
	By taking the complex dot product of the third equation of
	$\eqref{LinearF}$ with $\frac{\mu}{a}\hat{\phi}_{t}$, and retaining the real part, one has
	\begin{equation*}
	\frac{d}{dt}\left\{\frac{\mu D}{2a}|k|^2|\hat{\phi}|^2+\frac{b\mu}{2a}|\hat{\phi}|^2\right\}+\frac{\mu}{a}|\hat{\phi}_{t}|^2 =\mathrm{Re}(\hat{\rho}|\mu\hat{\phi_{t}}),
	\end{equation*}
	which along with  \eqref{LinearF1} implies
	\begin{eqnarray}\label{LinearF2}
	\frac{1}{2}\frac{d}{dt} \left\{\frac{P'(\bar{\rho})}{\bar{\rho}}|\hat{\rho}|^2+\bar{\rho}|\hat{u}|^2+\frac{\mu D}{a}|k|^2|\hat{\phi}|^2+\frac{b\mu}{a}|\hat{\phi}|^2-2\mathrm{Re}(\mu\hat{\phi}| \hat{\rho})\right\}+\alpha\bar{\rho}|\hat{u}|^2+\frac{\mu}{a}|\hat{\phi}_{t}|^2=0.
	\end{eqnarray}
	
	Secondly, by taking the complex dot product of the second equation of  $\eqref{LinearF}$ with $ik\hat{\rho}$, replacing
	$\partial_{t}\hat{\rho}$ by $-\bar{\rho}i k\cdot\hat{u}$ from the first equation of
	$\eqref{LinearF}$ and retaining the real part, one has
	\begin{equation}\label{LinearF3}
	\frac{d}{dt}\mathrm{Re}(\hat{u}|ik\hat{\rho})+\frac{P'(\bar{\rho})}{\bar{\rho}}|k|^2|\hat{\rho}|^2-\mu\mathrm{Re}(ik\hat{\phi}|ik\hat{\rho})=-\alpha\mathrm{Re}(\hat{u}|ik\hat{\rho})+\bar{\rho}|k\cdot\hat{u}|^2.
	\end{equation}
	Multiplying the third equation of $\eqref{LinearF}$ by $ik$, then taking the complex dot product of the resultant equation with $\frac{\mu}{a}ik\hat{\phi}$, and keeping the real part,  one has
	\begin{equation}\label{LinearF4}
	\frac{\mu}{2a}\frac{d}{dt}\left(|k|^2|\hat{\phi}|^2\right)+\frac{\mu D}{a}|k|^4|\hat{\phi}|^2+\frac{b\mu}{a}|k|^2|\hat{\phi}|^2=\mathrm{Re}\mu(ik\hat{\rho}|ik\hat{\phi}).
	\end{equation}
	Taking summation of \eqref{LinearF3} and \eqref{LinearF4} gives
	\begin{equation}\label{LinearF5}
	\begin{aligned}
	&\frac{d}{dt}\left\{\mathrm{Re}(\hat{u}|ik\hat{\rho})+\frac{\mu}{2a}|k|^2|\hat{\phi}|^2\right\}+\frac{\mu D}{a}|k|^4|\hat{\phi}|^2\\
	&+\frac{P'(\bar{\rho})}{\bar{\rho}}|k|^2|\hat{\rho}|^2+\frac{b\mu}{a}|k|^2|\hat{\phi}|^2-2\mathrm{Re}\mu(ik\hat{\rho}|ik\hat{\phi})\\
	=&-\alpha\mathrm{Re}(\hat{u}|ik\hat{\rho})+\bar{\rho}|k\cdot\hat{u}|^2.
	\end{aligned}
	\end{equation}
	Since $bP'(\bar{\rho})-a\mu\bar{\rho}>0$, the following matrix
	\begin{equation}\label{matrix}
	\left(
	\begin{array}{cc}
	\frac{P'(\bar{\rho})}{\bar{\rho}}& -\mu \\[3mm]
	-\mu  & \frac{b\mu}{a}
	\end{array}
	\right)
	\end{equation}
	is positive definite, which yields a  positive constant $C_{1}>0$ such that
	\begin{equation}\label{LinearF51}
	\frac{P'(\bar{\rho})}{\bar{\rho}}|k|^2|\hat{\rho}|^2+\frac{b\mu}{a}|k|^2|\hat{\phi}|^2-2\mathrm{Re}\mu(ik\hat{\rho}|ik\hat{\phi})\geq C_{1}\left(|k|^2|\hat{\rho}|^2+|k|^2|\hat{\phi}|^2\right).
	\end{equation}
	Applying \eqref{LinearF51} into \eqref{LinearF5} along with the Cauchy-Schwarz inequality imply
	\begin{equation*}
	\frac{d}{dt}\left\{\mathrm{Re}(\hat{u}|ik\hat{\rho})+\frac{\mu}{2a}|k|^2|\hat{\phi}|^2\right\}+\frac{\mu D}{a}|k|^4|\hat{\phi}|^2+\frac{C_{1}}{2}|k|^2|\hat{\rho}|^2+C_{1}|k|^2|\hat{\phi}|^2
	\leq C(1+|k|^2)|\hat{u}|^2,
	\end{equation*}
	which, multiplied by $1/(1+|k|^{2})$, gives
	\begin{equation}\label{Linear4}
	\frac{d}{dt}\left\{\frac{\mathrm{Re}(\hat{u}|ik\hat{\rho})}{1+|k|^2}+\frac{\mu}{2a}\frac{|k|^2}{1+|k|^2}|\hat{\phi}|^2\right\}+\frac{\mu D}{a}\frac{|k|^4}{1+|k|^2}|\hat{\phi}|^2+ \frac{\lambda |k|^2}{1+|k|^2}|[\hat{\rho},\hat{\phi}]|^2\leq C|\hat{u}|^2.
	\end{equation}
	
	Finally, let's define
	\begin{equation*}
	\begin{aligned}
	\mathcal{E}(\hat{U}(t,k))=&\frac{P'(\bar{\rho})}{\bar{\rho}}|\hat{\rho}|^2+\bar{\rho}|\hat{u}|^2+\frac{\mu D}{a}|k|^2|\hat{\phi}|^2+\frac{b\mu}{a}|\hat{\phi}|^2-2\mathrm{Re}(\mu \hat{\phi}| \hat{\rho})\\
	&+ \kappa\left\{\frac{\mathrm{Re}(\hat{u}|ik\hat{\rho})}{1+|k|^2}+\frac{\mu}{2a}\frac{|k|^2}{1+|k|^2}|\hat{\phi}|^2\right\}
	\end{aligned}
	\end{equation*}
	for a constant $0<\kappa\ll 1$ to be determined.
	Recall that the matrix \eqref{matrix} is positive definite.
	Then there exist two positive constants $C_{2}, C_{3}>0$ such that
	\begin{equation*}
	C_{2}\left(|\hat{\rho}|^2+|\hat{\phi}|^2\right)\leq \frac{P'(\bar{\rho})}{\bar{\rho}}|\hat{\rho}|^2+\frac{b\mu}{a}|\hat{\phi}|^2-2\mathrm{Re}(\mu \hat{\phi}| \hat{\rho})\leq C_{3}\left(|\hat{\rho}|^2+|\hat{\phi}|^2\right).
	\end{equation*}
	Notice that as long as $ 0<\kappa\ll 1$ is small enough, then $\mathcal{E}(\hat{U}(t,k))\sim |\hat{U}(t)|^{2}+|k|^2|\hat{\phi}|^2$ holds
	true and \eqref{equver} is proved. The sum of $\eqref{LinearF2}$ with $\eqref{Linear4}\times
	\kappa$ gives
	\begin{eqnarray}\label{DJF6}
	\partial_{t}\mathcal{E}(\hat{U}(t,k))+\lambda|\hat{u}|^2
	+\lambda\dfrac{|k|^{2}}{1+|k|^{2}}|[\hat{\rho},\hat{\phi}]|^2
	+\dfrac{\lambda|k|^{4}}{1+|k|^{2}}|\hat{\phi}|^{2}
	\leq 0.
	\end{eqnarray}
	Therefore,
	$\eqref{Lyainequlity1}$ follows from $\eqref{DJF6}$ by the observation
	\begin{eqnarray*}
		\lambda|\hat{u}|^2
		+\lambda\dfrac{|k|^{2}}{1+|k|^{2}}|[\hat{\rho},\hat{\phi}]|^2
		+\dfrac{\lambda|k|^{4}}{1+|k|^{2}}|\hat{\phi}|^{2}
		\geq \dfrac{\lambda|k|^{2}}{1+|k|^{2}}\left(|\hat{U}|^{2}+|k|^2|\hat{\phi}|^2\right).
	\end{eqnarray*}
	This completes the proof of Theorem \ref{lyapunov}.
\end{proof}

Theorem \ref{lyapunov} directly yields the pointwise
time-frequency estimate on  $|\hat{U}(t,k)|$ in terms of
initial data moduli $|\hat{U}_{0}(k)|$ and $|k||\hat{\phi}_{0}(k)|$.

\begin{corollary}
	Let $U(t,x)$, $t\geq0$, $x\in\mathbb{R}^{3}$ be a well-defined
	solution to the system $\eqref{HPVL}$-$\eqref{NLI}$. Then, there
	are $\lambda>0$, $ C>0$  such that
	\begin{eqnarray}\label{s4.j}
	|\hat{U}(t,k)|\leq C {\exp\left({-\tfrac{\lambda|k|^{2}t}{1+|k|^{2}}}\right)}(|\hat{U_{0}}(k)|+|k||\hat{\phi}_{0}|)
	\end{eqnarray}
	holds for any $t\geq 0$ and $k\in \mathbb{R}^{3}$.
\end{corollary}

Based on the pointwise
time-frequency estimate $\eqref{s4.j}$,  it is also straightforward to obtain the
$L^{p}$-$L^{q}$ time-decay property of the Cauchy problem
$\eqref{HPVL}$-$\eqref{NLI}$. Formally, the solution of
the Cauchy problem $\eqref{HPVL}$-$\eqref{NLI}$ is denoted by
\begin{eqnarray*}
	\begin{aligned}
		U(t)=\left[\rho,u,\phi\right]={e^{tL}}U_{0}, 
	\end{aligned}
\end{eqnarray*}
where ${e^{tL}} $ for $t\geq 0$ is called the linearized
solution operator corresponding to the system \eqref{HPVL}.

\begin{corollary}[see \cite{Duan} for instance]
	Let $1\leq p,r\leq 2\leq q\leq \infty$, $\ell\geq 0$ and let $m\geq
	1$ be an integer. Define
	\begin{equation}\label{thm.decay.1}
	\left[\ell+3\left(\frac{1}{r}-\frac{1}{q}\right)\right]_{+}
	=\left\{\begin{array}{ll}
	\ell, & \  \text{if $\ell$ is an integer and $r=q=2$},\\[3mm]
	{[}\ell+3(\frac{1}{r}-\frac{1}{q}){]}_{-}+1,  &\ \text{otherwise},
	\end{array}\right.
	\end{equation}
	where $[\cdot]_{-}$ denotes the integer part of  the argument.
	Then $e^{tL}$ satisfies
	the following time-decay property:
	\begin{eqnarray*}
		&&\begin{aligned}
			\|\nabla^{m}e^{Lt} U_0\|_{L^q}\leq&  C
			(1+t)^{-\tfrac{3}{2}(\tfrac{1}{p}-\tfrac{1}{q})-\tfrac{m}{2}}
			\|U_0\|_{L^p}+ C
			(1+t)^{-\tfrac{3}{2}(\tfrac{1}{p}-\tfrac{1}{q})-\tfrac{m+1}{2}}\|\phi_0\|_{L^p}\\
		&+ Ce^{-\lambda t}\|\nabla^{m+[3(\tfrac{1}{r}-\tfrac{1}{q})]_+}{U}_0\|_{L^r}+Ce^{-\lambda t}\|\nabla^{m+1+[3(\tfrac{1}{r}-\tfrac{1}{q})]_+}{\phi}_0\|_{L^r}
		\end{aligned}
	\end{eqnarray*}
	for any $t\geq 0$, where  $C=C(m,p,r,q,\ell)>0$ is a constant.
\end{corollary}

\section{Spectral analysis on the linearized system} \label{sec4}

In order to find a refined large-time asymptotic profile of solutions to \eqref{HPV1} with \eqref{1.2}, we conduct the spectral analysis for the linearized system \eqref{HPVL}-\eqref{NLI}.

\subsection{Preparations}
Let us first recall our linearized problem \eqref{HPVL}-\eqref{NLI}. In the rest of this section, for the simplicity of notations, we shall use $[\rho, u, \phi]$ to denote the solution  $[\rho_1, u, \phi_1]$ of $\eqref{HPVL}$-$\eqref{NLI}$, correspondingly, $ [\rho_{0}, u_{0}, \phi_{0}]$ to denote $[\rho_{1,0}, u_{0}, \phi_{1,0}]$ unless other stated. We now derive the asymptotic equations that one may expect in the large time.  By the idea of the asymptotic analysis, one may expect that the asymptotic profile of the linearized system \eqref{HPVL} satisfies
\begin{equation*}
\left\{
\begin{aligned}
&\partial_{t}\tilde{\rho}+\bar{\rho}\nabla \cdot \tilde{u}=0,\\
& \frac{P'(\bar{\rho})}{\bar{\rho}}\nabla\tilde{\rho}+\alpha\tilde{u}-\mu\nabla\tilde{\phi}=0,\\
&-a\tilde{\rho}+b\tilde{\phi}=0,
\end{aligned}\right.
\end{equation*}
with initial data
\begin{equation*}
\tilde{\rho}|_{t=0}=\rho_0.
\end{equation*}
Therefore, $\tilde{\rho}$, $\tilde{u}$ and $\tilde{\phi}$ are determined
according to the following equations 
\begin{eqnarray}\label{sol.rhob}
\left\{
\begin{aligned}
&\partial_{t}\tilde{\rho}-\dfrac{bP'(\bar{\rho})-a\mu\bar{\rho}}{b\alpha}\Delta
\tilde{\rho}=0,\\
&\begin{aligned}
\tilde{u}=&-\frac{bP'(\bar{\rho})-a\mu\bar{\rho}}{b\alpha\bar{\rho}}\nabla\rho,
\end{aligned}\\
&\tilde{\phi}=\frac{a}{b}\tilde{\rho},
\end{aligned}\right.
\end{eqnarray}
where initial data $\tilde{u}_{0}$ and $\tilde{\phi}_{0}$  are determined by $\tilde{\rho}_0$ in terms of the last two equations of \eqref{sol.rhob}, respectively. With the fact $bP'(\bar{\rho})-a\mu\bar{\rho}>0$, the first equation of \eqref{sol.rhob} is essentially a heat equation.  The diffusion wave $[\tilde{\rho},\tilde{u},\tilde{\phi}]$ defined in \eqref{def.pp3} is the solution of \eqref{sol.rhob}.  Its solution can also be expressed by the following Fourier transform:
\begin{eqnarray}
&&\label{rhobar}\hat{\tilde{\rho}}=\exp{\left(-\frac{bP'(\bar{\rho})-a\mu\bar{\rho}}{b\alpha}|k|^{2}t\right)}\hat{\rho}_{0},\ \ \ \ \ \ \ \ \ \ \ \ \ \ \ \ \ \ \ \ \ \ \ \ \ \ \ \ \
\end{eqnarray}
and
\begin{eqnarray}
&&\label{ubar}\hat{\tilde{u}}=\exp{\left(-\frac{bP'(\bar{\rho})-a\mu\bar{\rho}}{b\alpha}|k|^{2}t\right)}\left(-\frac{bP'(\bar{\rho})-a\mu\bar{\rho}}{b\alpha\bar{\rho}}\right)ik\hat{\rho}_{0},\\
&&\label{Ebar}\hat{\tilde{\phi}}=\exp{\left(-\frac{bP'(\bar{\rho})-a\mu\bar{\rho}}{b\alpha}|k|^{2}t\right)}\left(\frac{a}{b}\hat{\rho}_{0}\right).
\end{eqnarray}
The above expressions will be used later.

\subsection{Spectral representation}

\subsubsection{Asymptotic expansions and expressions}

In this subsection, we further explore the explicit solution  $U=[\rho,u,\phi]=e^{tL}U_{0}$ to the linearized Cauchy problem \eqref{HPVL}-\eqref{NLI}. Let us rewrite the system as two decoupled subsystems governing the time evolution of $[\rho, \nabla\cdot u, \phi]$  and $\nabla\times u$.
Taking the curl for the second equation of \eqref{HPVL}, one has
\begin{equation*}
\partial_{t}(\nabla\times u)+\alpha (\nabla\times u)=0.
\end{equation*}
In terms of the Fourier transform in $x$, one has
\begin{equation*}
\partial_{t}(ik \times \hat{u})+\alpha (ik\times \hat{u})=0,
\end{equation*}
with the initial data
\begin{equation*}
ik \times \hat{u}|_{t=0}=ik\times \hat{u}_{0}.
\end{equation*}
It is easy to obtain
\begin{equation*}
-\tilde{k}\times (\tilde{k} \times \hat{u})=e^{-\alpha t}\{-\tilde{k}\times (\tilde{k} \times \hat{u}_{0})\},
\end{equation*}
where $\tilde{k}=k/|k|$ for $|k|\neq 0$.

Taking the divergence for the second equation of \eqref{HPVL}, we get the equations of $[\rho,\nabla\cdot u, \phi]=:U_{\parallel}(t,x)$
\begin{equation*}
\left\{\begin{aligned}
& \pa_{t}\rho+\bar{\rho}\nabla\cdot u=0,\\
&\pa_t (\nabla \cdot u)+\frac{P'(\bar{\rho})}{\bar{\rho}}\Delta \rho+\alpha (\nabla\cdot u) -\mu \Delta\phi=0,\\
&\pa_t \phi-D\Delta \phi-a\rho+b\phi=0,
\end{aligned}\right.
\end{equation*}
with initial data
\begin{eqnarray*}
	[\rho,\nabla\cdot u,\phi]|_{t=0}=[\rho_{0},\nabla\cdot u_{0},\phi_{0}].
\end{eqnarray*}
Applying the Fourier transformation to the above equation, we have
\begin{equation}\label{LinearFD}
\left\{\begin{aligned}
& \pa_{t}\hat{\rho}+\bar{\rho}ik\cdot \hat{u}=0,\\
&\pa_t (ik\cdot \hat{u})-\frac{P'(\bar{\rho})}{\bar{\rho}}|k|^2\hat{\rho}+\alpha (ik\cdot \hat{u}) +\mu |k|^2\hat{\phi}=0,\\
&\pa_t \hat{\phi}+D|k|^2\hat{\phi}-a\hat{\rho}+b\hat{\phi}=0,
\end{aligned}\right.
\end{equation}
with the initial data
\begin{equation}\label{FluidFDI}
\hat{U}_{\parallel}(t,k)|_{t=0}=\hat{U}_{\parallel 0}(k):=[\hat{\rho}_{0},ik\cdot \hat{u}_{0}, \hat{\phi}_{0}].
\end{equation}
Then the solution to \eqref{LinearFD}-\eqref{FluidFDI} can be written as
\begin{equation*}
\hat{U}_{\parallel}(t,k)^T={e^{A(|k|)t}}\hat{U}_{\parallel 0}(k)^{T},
\end{equation*}
with matrix $A(|k|)$ defined by
\begin{equation*}
A(|k|)=:\left(\begin{array} {ccc}
0\ \  &  -\bar{\rho} &  0\\[3mm]
\frac{P'(\bar{\rho})}{\bar{\rho}}|k|^2\ \ & -\alpha  \ \ &  -\mu |k|^2\\[3mm]
a \ \ &0 & -b-D|k|^2
\end{array} \right).
\end{equation*}
By a direct computation, we see that the characteristic polynomial of
$A(|k|)$ is
\begin{eqnarray}
g(\lambda)=:\det(\lambda
I-A(|k|)) &=&\lambda^{3}+\underbrace{\left(b+D|k|^2+\alpha\right)}_{c_2}\lambda^2+\underbrace{\left(b\alpha+D\alpha|k|^2+ P'(\bar{\rho})|k|^2\right)}_{c_1}\lambda\notag\\[1.5mm]
&&+\underbrace{\left(b+D|k|^2\right)P'(\bar{\rho})|k|^2-a\mu\bar{\rho}|k|^2}_{c_0}\notag\\[1.5mm]
&=:&\lambda^{3}+c_{2}\lambda^2+c_{1}\lambda+c_{0}.\label{eigenpo}
\end{eqnarray}
One can find some elementary properties of the function $g(\lambda)$ as follows:
\begin{itemize}
\renewcommand{\itemsep}{5pt}
  \item $g(0)=\left(bP'(\bar{\rho})-a\mu\bar{\rho}\right)|k|^2+D|k|^2P'(\bar{\rho})|k|^2>0$ as $k\neq 0$;
  \item $g(-(b+\alpha+D|k|^2))=-(b+D|k|^2)^2\alpha-(b+D|k|^2)\alpha^2-P'(\bar{\rho})\alpha|k|^2-a\mu\bar{\rho}|k|^2<0;$
  \item $g'(\lambda)=3\lambda^{2}+2\left(b+D|k|^2+\alpha\right)\lambda+\left(b\alpha+D\alpha|k|^2+ P'(\bar{\rho})|k|^2\right)>0
$ for $\lambda\geq 0;$
\item $g'(\lambda)=\lambda^{2}+2(\lambda+b+D|k|^2+\alpha)\lambda+\left(b\alpha+D\alpha|k|^2+ P'(\bar{\rho})|k|^2\right)\geq b\alpha+D\alpha|k|^2+ P'(\bar{\rho})|k|^2>0,$
for $\lambda\leq -(b+\alpha+D|k|^2);$
\item $g(\lambda)$ is strictly increasing for $\lambda\leq
-(b+\alpha+D|k|^2)$ or $\lambda\geq 0$.
\end{itemize}
The above properties imply that the equation $g(\lambda)=0$ has at
least one negative real root lying  in $\left(-(b+\alpha+D|k|^2),0\right)$. We can distinguish several
possible cases for the roots by using the discriminant,
\begin{equation*}
\Delta=18c_{2}c_{1}c_{0}-4c_{2}^{3}c_{0}+c_{2}^{2}c_{1}^{2}-4c_{1}^{3}-27c_{0}^{2}.
\end{equation*}
\begin{itemize}
  \item  $\Delta>0$, then $g(\lambda)=0$ has three distinct real
  roots;
  \item  $\Delta<0$, then $g(\lambda)=0$ has one real root and non-real two complex conjugate roots;
 \item  $\Delta=0$, then $g(\lambda)=0$ has multiple roots which are all real.
\end{itemize}
Here the term with the highest power of $\Delta$ is $-4D^4P'(\bar{\rho})|k|^{10}$, which implies that $\Delta$ is a polynomial of $|k|$ with degree $10$.  Hence there exist at most finite number of values of  $|k|$ such that $\Delta=0$. Hereafter, we exclude these finite number of values of  $|k|$ since they will not affect the $L^{p}$-estimates of solutions to the linearized equations. We analyze the roots of the equation
$g(\lambda)=\det(\lambda
I-A(|k|))=0$ and their asymptotic properties as $|k|\rightarrow 0$.
Clearly the eigenvalues $\lambda_i (i=1,2,3)$ of $A(|k|)$ satisfy
\begin{eqnarray*}
	\begin{aligned}
		&\lambda_{1}+\lambda_{2}+\lambda_{3}=-b-D|k|^2-\alpha,\\
		&\lambda_{1}\lambda_{2}\lambda_{3}=-\left(b+D|k|^2\right)P'(\bar{\rho})|k|^2+a\mu\bar{\rho}|k|^2,\\
		&\lambda_{1}\lambda_{2}+\lambda_{1}\lambda_{3}+\lambda_{2}\lambda_{3}=b\alpha+D\alpha|k|^2+ P'(\bar{\rho})|k|^2.
	\end{aligned}
\end{eqnarray*}
The perturbation theory (see \cite{Ho} or \cite{Kato1}) for one-parameter family of
matrix $A(|k|)$ for $|k|\rightarrow 0$ implies that
$\lambda_{j}(|k|) (j=1,2,3)$ has the following asymptotic expansions:
\begin{equation*}
\lambda_{j}(|k|)=\sum_{\ell=0}^{+\infty}\lambda_{j}^{(\ell)}|k|^{\ell},
\end{equation*}
where $\lambda_{j}^{(\ell)}$ is the coefficient of $|k|^{\ell}$ in the expansion. Notice that $\lambda_{j}^{(0)}$ are the roots of the following
equation:
\begin{equation*}
\left[\lambda_{j}^{(0)}\right]^3+(b+\alpha)\left[\lambda_{j}^{(0)}\right]^2+b\alpha\lambda_{j}^{(0)}=0.
\end{equation*}
Then we have
\begin{equation*}
\lambda_{1}^{(0)}=0,\ \ \ \ \ \lambda_{2}^{(0)}=-b<0,\ \ \ \ \ \lambda_{3}^{(0)}=-\alpha<0.
\end{equation*}
By straightforward computations along with \eqref{eigenpo}, we find that
\begin{equation}\label{roots}
\begin{aligned}
&\lambda_{1}(|k|)=-\dfrac{bP'(\bar{\rho})-a\mu\bar{\rho}}{b\alpha}|k|^{2}+O(|k|^{4}),\\
&\lambda_{2}(|k|)=-b+O(|k|),\\
&\lambda_{3}(|k|)=-\alpha+O(|k|),
\end{aligned}
\end{equation}
which imply that $\lambda_{j}(|k|)$ are distinct to each other as $|k|\rightarrow 0$.

Next we give the asymptotic expressions of ${e^{A(|k|)t}}$ as
$|k|\rightarrow 0$. We note that the solution matrix
${e^{A(|k|)t}}$ has the spectral decomposition
\begin{equation*}
{e^{A(|k|)t}}=\sum_{j=1}^{3}{\exp({\lambda_{j}(|k|)t})}P_{j}(|k|),
\end{equation*}
where $\lambda_{j}(|k|)$ are the eigenvalues of $A(|k|)$ and
$P_{j}(|k|)$ are the corresponding eigenprojections. Notice that $\lambda_{j}(|k|)$ are distinct to each other as $|k|\rightarrow 0$. Then $P_{j}(|k|)$ can be written as
\begin{equation*}
\displaystyle P_{j}(|k|)=\prod_{\ell\neq
	j}\frac{A(|k|)-\lambda_{\ell}(|k|)I}{\lambda_{j}(|k|)-\lambda_{\ell}(|k|)}, \ j=1,2,3.
\end{equation*}

We estimate $P_{1}(|k|)$ as
\begin{equation*}
\begin{aligned}
\displaystyle P_{1}(|k|)=&\frac{A(|k|)-\lambda_{2}(|k|)I}{\lambda_{1}(|k|)-\lambda_{2}(|k|)}\cdot\frac{A(|k|)-\lambda_{3}(|k|)I}{\lambda_{1}(|k|)-\lambda_{3}(|k|)}\\
=&\frac{A(|k|)^2-(\lambda_{2}(|k|)+\lambda_{3}(|k|))A(|k|)+(\lambda_{2}(|k|)\lambda_{3}(|k|)I}{\lambda_{1}(|k|)^2-(\lambda_{2}(|k|)+\lambda_{3}(|k|))\lambda_{1}(|k|) +(\lambda_{2}(|k|)\lambda_{3}(|k|)}=:\frac{(f_{ij})_{3\times 3}}{P^{\mathrm{den}}_{1}}.
\end{aligned}
\end{equation*}
One can compute
\begin{equation*}
[A(|k|)]^{2}=\left(\begin{array} {ccc}
-P'(\bar{\rho})|k|^2 & \alpha\bar{\rho} &  \mu\bar{\rho}|k|^2\\[3mm]
\ -\frac{\alpha P'(\bar{\rho})}{\bar{\rho}}|k|^2-a\mu|k|^2\  &\   -P'(\bar{\rho})|k|^2+\alpha^2\  & \   (b+\alpha) \mu |k|^2+D \mu |k|^4\ \\[3mm]
-ab-aD|k|^2 & -a\bar{\rho} & b^2+2bD|k|^2+D^2|k|^4
\end{array} \right),
\end{equation*}
and
\begin{eqnarray}\label{root23}
\begin{aligned}
&\lambda_{2}+\lambda_{3}=-b-D|k|^2-\alpha-\lambda_{1}=-(b+\alpha)-\left(D-\dfrac{bP'(\bar{\rho})-a\mu\bar{\rho}}{b\alpha}\right)|k|^{2}+O(|k|^4),\\
&\begin{aligned}
\lambda_{2}\lambda_{3}=&b\alpha+D\alpha|k|^2+ P'(\bar{\rho})|k|^2-\lambda_{1}\left(\lambda_{2}+\lambda_{3}\right)\\
=&b\alpha+\left(D\alpha+P'(\bar{\rho})-(b+\alpha)\dfrac{bP'(\bar{\rho})-a\mu\bar{\rho}}{b\alpha}\right)|k|^2+O(|k|^4).
\end{aligned}
\end{aligned}
\end{eqnarray}

Let us compute $f_{ij}$ ($1\leq i,j\leq 3)$ as follows. For $f_{11}$, one has
\begin{eqnarray*}
	\begin{aligned}
		f_{11}
		=& -P'(\bar{\rho})|k|^2+b\alpha+\left(D\alpha+P'(\bar{\rho})-(b+\alpha)\dfrac{bP'(\bar{\rho})-a\mu\bar{\rho}}{b\alpha}\right)|k|^2+O(|k|^4)\\
		=&b\alpha+\left(D\alpha-(b+\alpha)\dfrac{bP'(\bar{\rho})-a\mu\bar{\rho}}{b\alpha}\right)|k|^2+O(|k|^4).
	\end{aligned}
\end{eqnarray*}
In a similar way, we can get
\begin{eqnarray*}
	\begin{aligned}
		&\begin{aligned}
			f_{12}=&\alpha\bar{\rho}-(\lambda_{2}+\lambda_{3})(-\bar{\rho})\\
			=&\alpha\bar{\rho}-(-(b+\alpha)-\left(D-\dfrac{bP'(\bar{\rho})-a\mu\bar{\rho}}{b\alpha}\right)|k|^{2}+O(|k|^4))(-\bar{\rho})\\
			=&-b\bar{\rho}-\left(D-\dfrac{bP'(\bar{\rho})-a\mu\bar{\rho}}{b\alpha}\right)\bar{\rho}|k|^{2}+O(|k|^4),
		\end{aligned}\\
		& f_{13}=\mu\bar{\rho}|k|^2.
	\end{aligned}
\end{eqnarray*}

For $f_{21}$, one has
\begin{eqnarray*}
	\begin{aligned}
		f_{21}= &-\frac{\alpha P'(\bar{\rho})}{\bar{\rho}}|k|^2-a\mu|k|^2-\left(-(b+\alpha)-\left(D-\dfrac{bP'(\bar{\rho})-a\mu\bar{\rho}}{b\alpha}\right)|k|^{2}+O(|k|^4)\right)\frac{P'(\bar{\rho})}{\bar{\rho}}|k|^2\\
		=& \frac{bP'(\bar{\rho})-a\mu\bar{\rho}}{\bar{\rho}}|k|^2+O(|k|^4),
	\end{aligned}
\end{eqnarray*}
and
\begin{eqnarray*}
	\begin{aligned}
		f_{22}=&-P'(\bar{\rho})|k|^2+\alpha^2-\left(-(b+\alpha)-\left(D-\dfrac{bP'(\bar{\rho})-a\mu\bar{\rho}}{b\alpha}\right)|k|^{2}+O(|k|^4)\right)(-\alpha)\\
		&+b\alpha+\left(D\alpha+P'(\bar{\rho})-(b+\alpha)\dfrac{bP'(\bar{\rho})-a\mu\bar{\rho}}{b\alpha}\right)|k|^2+O(|k|^4)\\
		=&-\dfrac{bP'(\bar{\rho})-a\mu\bar{\rho}}{\alpha}|k|^2+O(|k|^4).
	\end{aligned}
\end{eqnarray*}
Similarly, one has
\begin{eqnarray*}
	\begin{aligned}
		f_{23}=& (b+\alpha) \mu |k|^2+D \mu |k|^4-\left(-(b+\alpha)-\left(D-\dfrac{bP'(\bar{\rho})-a\mu\bar{\rho}}{b\alpha}\right)|k|^{2}+O(|k|^4)\right)(-\mu |k|^2)\\
		=&\dfrac{bP'(\bar{\rho})-a\mu\bar{\rho}}{b\alpha}\mu|k|^{4}+O(|k|^6).
	\end{aligned}
\end{eqnarray*}
Moreover, it holds that
\begin{eqnarray*}
	\begin{aligned} f_{31}
		=& -ab-aD|k|^2-\left(-(b+\alpha)-\left(D-\dfrac{bP'(\bar{\rho})-a\mu\bar{\rho}}{b\alpha}\right)|k|^{2}+O(|k|^4)\right)a\\
		=&a\alpha-a \dfrac{bP'(\bar{\rho})-a\mu\bar{\rho}}{b\alpha}|k|^2+O(|k|^4),
	\end{aligned}
\end{eqnarray*}
and
\begin{eqnarray*}
	&&f_{32}=-a\bar{\rho},\\
	&&\begin{aligned}
		f_{33}=&b^2+2bD|k|^2+D^{2}|k|^4-\left(-(b+\alpha)-\left(D-\dfrac{bP'(\bar{\rho})-a\mu\bar{\rho}}{b\alpha}\right)|k|^{2}+O(|k|^4)\right)(-b-D|k|^2),\\
		&+b\alpha+\left(D\alpha+P'(\bar{\rho})-(b+\alpha)\dfrac{bP'(\bar{\rho})-a\mu\bar{\rho}}{b\alpha}\right)|k|^2+O(|k|^4)\\
		=&\frac{a\mu\bar{\rho}}{b}|k|^{2}+O(|k|^4).
	\end{aligned}
\end{eqnarray*}
Because of \eqref{root23}, we denote $P^{\mathrm{den}}_{1}=\sum\limits_{\ell=0}^{+\infty}g^{(2\ell)}|k|^{2\ell}$. Then
\begin{eqnarray*}
	\frac{1}{P^{\mathrm{den}}_{1}}=\frac{1}{b\alpha+g^{(2)}|k|^2+O(|k|^4)}=\frac{1}{b\alpha}-\frac{g^{(2)}}{[g^{(0)}]^2}|k|^2+O(|k|^4).
\end{eqnarray*}

Let $P_{j}^{1}(ik),\ P_{j}^{2}(ik),\ P_{j}^{3}(ik)$
be the three row vectors of $P_{j}(ik)$, $j=1,2,3$.
Then we have the expressions of $\hat{\rho}$,
$\hat{u}$ and $\hat{\phi}$
for $|k|\rightarrow 0$ as follows:
\begin{eqnarray*}
	\begin{aligned}
		\hat{\rho}
		=&\frac{\exp{(\lambda_{1}(ik)t)}}{P^{\mathrm{den}}_{1}} \left(f_{11}\hat{\rho}_{0}+f_{12}ik\cdot \hat{u}_{0}+f_{13}\hat{\phi}_{0}\right)+\sum_{j=2}^{3}\exp{(\lambda_{j}(ik)t)}P_{j}^{1}(ik)\hat{U}_{\parallel
			0}(k)^{T}\\
		=&\exp{(\lambda_{1}(ik)t)}\hat{\rho}_{0}+O(|k|)\exp{(\lambda_{1}(ik)t)}\left|\hat{U}_{0}(k)\right|+\sum_{j=2}^{3}\exp{(\lambda_{j}(ik)t)}P_{j}^{1}(ik)\hat{U}_{\parallel
			0}(k)^{T},
	\end{aligned}
\end{eqnarray*}
\begin{eqnarray*}
	\begin{aligned}
		\hat{u}=&-\tilde{k}\times \tilde{k}\times \hat{u}-\frac{ik}{|k|^2}ik\cdot \hat{u}\\
		=& \exp{(-\alpha t)} (-\tilde{k}\times\tilde{k}\times\hat{u}_{0})-\frac{ik}{|k|^2}
		\bigg\{\frac{\exp{(\lambda_{1}(ik)t)}}{P^{\mathrm{den}}_{1}}\left(f_{21}\hat{\rho}_{0}+f_{22}ik\cdot \hat{u}_{0}+f_{23}\hat{\phi}_{0}\right)
		\\
		&\ \ \ \ \ \ \ \ \ \ \ \ \ \ \ \ \ \ \ \ \ \ \ \ \ \ \ \ \ \ \ +\sum_{j=2}^{3}\exp{(\lambda_{j}(ik)t)}P_{j}^{2}(ik)\hat{U}_{\parallel
			0}(k)^T\bigg\}\\
		=& \exp{(-\alpha t)} (-\tilde{k}\times\tilde{k}\times\hat{u}_{0})+\exp{(\lambda_{1}(ik)t)}\left(- \frac{bP'(\bar{\rho})-a\mu\bar{\rho}}{b\alpha\bar{\rho}}ik\hat{\rho}_{0}\right)\\
		&\quad\quad\quad+O(|k|^{2})\exp{(\lambda_{1}(ik)t)}|\hat{U}_{0}(k)|+\sum_{j=2}^{3}\exp{(\lambda_{j}(ik)t)}P_{j}^{2}(ik)\hat{U}_{\parallel
			0}(k)^T,
	\end{aligned}
\end{eqnarray*}
\begin{eqnarray*}
	\begin{aligned}
		\hat{\phi}=
		&\frac{\exp{(\lambda_{1}(ik)t)}}{P^{\mathrm{den}}_{1}} \left(f_{31}\hat{\rho}_{0}+f_{32}ik\cdot \hat{u}_{0}+f_{33}\hat{\phi}_{0}\right)+\sum_{j=2}^{3}\exp{(\lambda_{j}(ik)t)}P_{j}^{3}(ik)\hat{U}_{\parallel
			0}(k)^{T}\\
		=&\exp{(\lambda_{1}(ik)t)}\left(\frac{a}{b}\hat{\rho}_{0}\right)+O(|k|)\exp{(\lambda_{1}(ik)t)}|\hat{U}_{0}(k)|+\sum_{j=2}^{3}\exp{(\lambda_{j}(ik)t)}P_{j}^{3}(ik)\hat{U}_{\parallel
			0}(k)^T.
	\end{aligned}
\end{eqnarray*}

%

\subsubsection{Error estimates}
\begin{lemma}\label{errorL}
	There is $r_{0}>0$ such that for $|k|\leq r_{0}$ and $t\geq 0$, the error term $|\hat{U}-\hat{\tilde{U}}|$ can be bounded as
	\begin{eqnarray}
	&&|\hat{\rho}(t,k)-\hat{\tilde{\rho}}(t,k)|\leq C
	|k|\exp{\left(-\lambda|k|^{2}t\right)}\left|\hat{U}_{0}(k)\right|+C\exp{\left(-\lambda t\right)}\left|\hat{U}_{0}(k)\right|,\label{error1}\\
	&&|\hat{u}(t,k)-\hat{\tilde{u}}(t,k)|\leq C
	|k|^{2}\exp{\left(-\lambda|k|^{2}t\right)}\left|\hat{U}_{
		0}(k)\right|+C\exp{\left(-\lambda
		t\right)}\left|\hat{U}_{
		0}(k)\right|,\label{error2}\\
	&&|\hat{\phi}(t,k)-\hat{\tilde{\phi}}(t,k)|\leq C
	|k|\exp{\left(-\lambda|k|^{2}t\right)}\left|\hat{U}_{0}(k)\right|+C\exp{\left(-\lambda
		t\right)}\left|\hat{U}_{0}(k)\right|,\label{error3}
	\end{eqnarray}
	where $C$ and $\lambda$ are positive constants.
\end{lemma}

\begin{proof} It follows from the expressions of
	$\hat{\rho}$ and $\hat{\tilde{\rho}}$ that
	\begin{eqnarray*}
		&&\begin{aligned}
			&\hat{\rho}(t,k)-\hat{\tilde{\rho}}(t,k)\\
			= &\exp{(\lambda_{1}(|k|)t)}\hat{\rho}_{0}-
			\exp{\left(-\frac{bP'(\bar{\rho})-a\mu\bar{\rho}}{b\alpha}|k|^{2}t\right)}\hat{\rho}_{0}\\
			&+O(|k|)\exp{(\lambda_{1}(|k|)t)}\left|\hat{U}_{0}(k)\right|+\sum_{j=2}^{3}\exp{(\lambda_{j}(|k|)t)}P_{j}^{1}(|k|)\hat{U}_{\parallel
				0}(k)^{T}\\
			:=&\hat{R}_{11}(|k|)+\hat{R}_{12}(|k|)+\hat{R}_{13}(|k|).
		\end{aligned}
	\end{eqnarray*}
	We have from \eqref{roots} that
	\begin{equation*}
	\lambda_{1}(|k|)+\frac{bP'(\bar{\rho})-a\mu\bar{\rho}}{b\alpha}|k|^{2}=O(|k|^{4}),
	\end{equation*}
	and
	\begin{eqnarray*}
		\begin{aligned}
			&\left|\exp{(\lambda_{1}(|k|)t)}-
			\exp{\left(-\frac{bP'(\bar{\rho})-a\mu\bar{\rho}}{b\alpha}|k|^{2}t\right)}\right|\\
			&=\exp{\left(-\frac{bP'(\bar{\rho})-a\mu\bar{\rho}}{b\alpha}|k|^{2}t\right)}
			\left|\exp{\left(\lambda_{1}(|k|)t+\frac{bP'(\bar{\rho})-a\mu\bar{\rho}}{b\alpha}|k|^{2}t\right)}-1\right|\\
			&\leq
			C\exp{\left(-\frac{bP'(\bar{\rho})-a\mu\bar{\rho}}{b\alpha}|k|^{2}t\right)}
			|k|^{4}t \exp{(C|k|^{4}t)}\\
			&\leq  C|k|^{2}\exp{\left(-\lambda|k|^{2}t\right)},
		\end{aligned}
	\end{eqnarray*}
	as $|k|\to 0$.
	Therefore, we obtain that
	\begin{eqnarray*}
		\left|\hat{R}_{11}(|k|)\right|\leq
		C|k|^{2}\exp{\left(-\lambda|k|^{2}t\right)}\left|\hat{\rho}_{0}(k)\right|\ \ \ \mathrm{as}\ \ |k|\rightarrow 0.
	\end{eqnarray*}
	Note that $\lambda_{1}(|k|)\leq -\lambda|k|^2$ and
	$|{\exp({\lambda_{1}(|k|)t})}|\leq {\exp({-\lambda|k|^2t})}$
	as $|k|\rightarrow 0$. Consequently, we find that
	\begin{eqnarray*}
		\left|\hat{R}_{12}(|k|)\right|\leq C|k|
		{\exp({-\lambda|k|^2t})}\left|\hat{U}_{0}(k)\right|\quad \mathrm{as}\ \ |k|\rightarrow 0.
	\end{eqnarray*}
	Now it suffices to estimate $\left|\hat{R}_{13}(|k|)\right|$. Recall that
	\eqref{roots} gives ${\exp({\lambda_{j}(ik)t})}\leq
	e^{-\lambda t}$  $(j=2,3)$ as $|k|\rightarrow 0$. Also notice
	$P_{j}^{1}(ik)=O(1)$ $(j=2,3)$. Thus we have
	\begin{eqnarray*}
		\left|\hat{R}_{13}(ik)\right|\leq C{e^{-\lambda t}}\left|\hat{U}_{0}(k)\right|\ \ \ \mathrm{as}\ \
		|k|\rightarrow 0.
	\end{eqnarray*}
	This yields the desired estimate $\eqref{error1}$.
	
	In a similar way, we can prove \eqref{error2} and \eqref{error3}, and complete the proof of Lemma \ref{errorL}.
\end{proof}

Next, we consider the properties of $\hat{\rho}(t,k)$,
$\hat{u}(t,k)$ and $\hat{\phi}(t,k)$
as  $|k|\rightarrow \infty$. It follows from \eqref{s4.j} that
\begin{eqnarray}\label{s4.je}
|\hat{U}(t,k)|\leq \left\{\begin{aligned}
&C {\exp({- \lambda|k|^{2}t })}(|\hat{U}_{0}(k)|+|k||\hat{\phi}_{0}|),\  \  |k|\leq r_{0},\\
&C {\exp({- \lambda t )}}(|\hat{U}_{0}(k)|+|k||\hat{\phi}_{0}(k)|),\ \ \  |k|\geq
r_{0}.
\end{aligned}
\right.
\end{eqnarray}
Here $r_{0}$ is defined in Lemma \ref{errorL}. Combining
\eqref{s4.je} with \eqref{rhobar}, \eqref{ubar} and \eqref{Ebar}, we have the following pointwise estimate for the error terms.

\begin{lemma}\label{errorL2}
	Let  $r_{0}>0$ be given in Lemma \ref{errorL}. For $|k|\geq r_0$ and $t\geq 0$, the error $|\hat{U}-\hat{\tilde{U}}|$ can be bounded as
	\begin{eqnarray*}
		&&|\hat{\rho}(t,k)-\hat{\tilde{\rho}}(t,k)|\leq C
		{\exp({- \lambda t )}}(|\hat{U}_{0}(k)|+|k||\hat{\phi}_{0}(k)|),
		\\
		&&|\hat{u}(t,k)-\hat{\tilde{u}}(t,k)|\leq  {\exp({- \lambda t )}}(|\hat{U}_{0}(k)|+|k||\hat{\phi}_{0}|)+C\exp{\left(-\lambda t\right)} |k||\hat{\rho}_{0}(k)|,
		\\
		&&|\hat{\phi}(t,k)-\hat{\tilde{\phi}}(t,k)|\leq C
		{\exp({- \lambda t )}}(|\hat{U}_{0}(k)|+|k||\hat{\phi}_{0}(k)|),
	\end{eqnarray*}
	where $C$ and $\lambda$ are positive constants.
\end{lemma}

Based on Lemmas \ref{errorL}-\ref{errorL2} and \cite[Theorem 4.2]{Duan}, the time-decay
properties for the difference terms $\rho-\tilde{\rho}$,
$u-\tilde{u}$ and
$\phi-\tilde{\phi}$ are stated as follows.

\begin{proposition}\label{thm.decaypar}
	Let $1\leq p,r,s\leq 2\leq q\leq \infty$, and let $m\geq
	0$ be an integer. Suppose that
	$[\rho,u,\phi]$ is the solution
	to the Cauchy problem \eqref{HPVL}-\eqref{NLI}. Then
	$U-\tilde{U}=[\rho-\tilde{\rho},u-\tilde{u},\phi-\tilde{\phi}]$ satisfies the following time-decay property:
	\begin{multline*}
	\|\nabla ^{m}(\rho(t)-\tilde{\rho}(t))\|_{L^q}\leq
	C(1+t)^{-\frac{3}{2}(\frac{1}{p}-\frac{1}{q})-\frac{m+1}{2}}\|U_{0}\|_{L^{p}}+Ce^{-\lambda t}\|\nabla^mU_{0}\|_{L^{s}}\\
	+Ce^{-\lambda t}
	\|\nabla^{m+[3(\frac{1}{r}-\frac{1}{q})]_+}U_{0}\|_{L^r}+Ce^{-\lambda t}\|\nabla^{m+1+[3(\frac{1}{r}-\frac{1}{q})]_+}\phi_{0}\|_{L^{r}},
	\end{multline*}
	\begin{multline*}
	\|\nabla
	^{m}(u(t)-\tilde{u}(t))\|_{L^q}\leq
	C(1+t)^{-\frac{3}{2}(\frac{1}{p}-\frac{1}{q})-\frac{m+2}{2}}\|U_{0}\|_{L^{p}}+Ce^{-\lambda t}\|\nabla^mU_{0}\|_{L^{s}}\\
	+Ce^{-\lambda t}
	\|\nabla^{m+[3(\frac{1}{r}-\frac{1}{q})]_+}U_{0}\|_{L^r}+Ce^{-\lambda t}\|\nabla^{m+1+[3(\frac{1}{r}-\frac{1}{q})]_+}[\rho_{0,}\phi_{0}]\|_{L^{r}},
	\end{multline*}
	and
	\begin{multline*}
	\|\nabla ^{m}(\phi(t)-\tilde{\phi}(t))\|_{L^q}\leq
	C(1+t)^{-\frac{3}{2}(\frac{1}{p}-\frac{1}{q})-\frac{m+1}{2}}\|U_{0}\|_{L^{p}}+Ce^{-\lambda t}\|\nabla^mU_{0}\|_{L^{s}}\\
	+Ce^{-\lambda t}
	\|\nabla^{m+[3(\frac{1}{r}-\frac{1}{q})]_+}U_{0}\|_{L^r}+Ce^{-\lambda t}\|\nabla^{m+1+[3(\frac{1}{r}-\frac{1}{q})]_+}\phi_{0}\|_{L^{r}},
	\end{multline*}
	for any $t\geq 0$, where  $C=C(m,p,r,q,\ell)$ and
	$[\ell+3(\tfrac{1}{r}-\tfrac{1}{q})]_+$ is defined in
	\eqref{thm.decay.1}.
\end{proposition}

Combining the expressions of $\hat{\rho}(t,k)$,
$\hat{u}(t,k)$ and $\hat{\phi}(t,k)$ as  $|k|\rightarrow 0$  with
\eqref{s4.je}, we
have the following $L^{p}-L^{q}$ time-decay estimates of $[\rho,u,\phi]$.

\begin{proposition}\label{thm.decayrup}
	Let $1\leq p,r,s\leq 2\leq q\leq \infty$, and let $m\geq
	0$ be an integer. Suppose that
	$[\rho,u,\phi]$ is the solution
	to the Cauchy problem \eqref{HPVL}-\eqref{NLI}. Then
	$U=[\rho,u,\phi]$ satisfies the following time-decay property:
	\begin{multline*}
	\|\nabla ^{m}\rho(t)\|_{L^q}\leq C(1+t)^{-\frac{3}{2}(\frac{1}{p}-\frac{1}{q})-\frac{m}{2}}\|\rho_{0}\|_{L^{p}}+
	C(1+t)^{-\frac{3}{2}(\frac{1}{p}-\frac{1}{q})-\frac{m+1}{2}}\|U_{0}\|_{L^{p}}+Ce^{-\lambda t}\|\nabla^mU_{0}\|_{L^{s}}\\
	+Ce^{-\lambda t}
	\|\nabla^{m+[3(\frac{1}{r}-\frac{1}{q})]_+}U_{0}\|_{L^r}+Ce^{-\lambda t}\|\nabla^{m+1+[3(\frac{1}{r}-\frac{1}{q})]_+}\phi_{0}\|_{L^{r}},
	\end{multline*}
	\begin{multline*}
	\|\nabla
	^{m}u(t)\|_{L^q}\leq C(1+t)^{-\frac{3}{2}(\frac{1}{p}-\frac{1}{q})-\frac{m+1}{2}}\|\rho_{0}\|_{L^{p}}
	+C(1+t)^{-\frac{3}{2}(\frac{1}{p}-\frac{1}{q})-\frac{m+2}{2}}\|U_{0}\|_{L^{p}}+Ce^{-\lambda t}\|\nabla^mU_{0}\|_{L^{s}}\\
	+Ce^{-\lambda t}
	\|\nabla^{m+[3(\frac{1}{r}-\frac{1}{q})]_+}U_{0}\|_{L^r}+Ce^{-\lambda t}\|\nabla^{m+1+[3(\frac{1}{r}-\frac{1}{q})]_+}[\rho_{0},\phi_{0}]\|_{L^{r}},
	\end{multline*}
	and
	\begin{multline*}
	\|\nabla ^{m}\phi(t)\|_{L^q}\leq C(1+t)^{-\frac{3}{2}(\frac{1}{p}-\frac{1}{q})-\frac{m}{2}}\|\rho_{0}\|_{L^{p}}+
	C(1+t)^{-\frac{3}{2}(\frac{1}{p}-\frac{1}{q})-\frac{m+1}{2}}\|U_{0}\|_{L^{p}}+Ce^{-\lambda t}\|\nabla^mU_{0}\|_{L^{s}}\\
	+Ce^{-\lambda t}
	\|\nabla^{m+[3(\frac{1}{r}-\frac{1}{q})]_+}U_{0}\|_{L^r}+Ce^{-\lambda t}\|\nabla^{m+1+[3(\frac{1}{r}-\frac{1}{q})]_+}\phi_{0}\|_{L^{r}},
	\end{multline*}
	for any $t\geq 0$, where  $C=C(m,p,r,q,\ell)$ and
	$[\ell+3(\tfrac{1}{r}-\tfrac{1}{q})]_+$ is defined in
	\eqref{thm.decay.1}.
\end{proposition}

\section{Asymptotic behaviour of the nonlinear system}\label{sec5}
We are devoted to proving Theorem \ref{thm.GE} in this section.
\subsection{Global existence}
Before exploiting the large-time  profile of solutions, we first establish the global existence of solutions to \eqref{HPV1} with \eqref{1.2}. To this end, we set $\rho_2=\rho-\bar{\rho}$, $\phi_2=\phi-\bar{\phi}$ and reformulate the problem \eqref{HPV1} with \eqref{1.2} around the constant equilibrium $[\bar{\rho},0,\bar{\phi}]$ with $\bar{\phi}=\frac{a}{b}\bar{\rho}$ as
\begin{equation}\label{HPV3}
\left\{\begin{aligned}
& \pa_{t}\rho_2+\bar{\rho}\nabla\cdot u=g_{1},\\
&\pa_t u+\frac{P'(\bar{\rho})}{\bar{\rho}}\nabla \rho_2+\alpha u -\mu\nabla\phi_2=g_{2},\\
&\pa_t \phi_2-D\Delta \phi_2-a\rho_2+b\phi_2=0,
\end{aligned}\right.
\end{equation}
where the nonlinear terms $g_{1}$ and $g_{2}$  are defined as follows
\begin{equation}\label{sec5.ggg}
\arraycolsep=1.5pt \left\{
\begin{aligned}
&g_{1}=-\nabla\cdot f_1, \ f_{1}=\rho_2 u\\
&g_{2}=- u \cdot \nabla u -\left(\frac{P'(\rho_2+\bar{\rho})}{\rho_2+\bar{\rho}}-\frac{P'(\bar{\rho})}{\bar{\rho}}\right)\nabla
\rho_2.
\end{aligned}\right.
\end{equation}
The initial data are given by
\begin{eqnarray}\label{NI}
[\rho_2,u,\phi_2]|_{t=0}=[\rho_{2,0},u_{0},\phi_{2,0}]=[\rho_0-\bar{\rho},u_{0},\phi_{0}-\bar{\phi}].
\end{eqnarray}

Next we shall focus on the reformulated problem \eqref{HPV3}-\eqref{NI} and first explore the global existence of solutions. Without confusion, in the rest of this section, we still use $[\rho,u,\phi]$ to denote $[\rho_2,u,\phi_2]$, correspondingly, $ [\rho_{0}, u_{0}, \phi_{0}]$ to denote $[\rho_{2,0}, u_{0}, \phi_{2,0}]$ for simplicity unless otherwise stated.  The main result of this section about the global existence of solutions to the reformulated Cauchy problem
\eqref{HPV3}-\eqref{NI} with small smooth initial data are stated as follows.


\begin{theorem}\label{pro.2.1}
	Given $\bar{\rho}>0$ and $N\ge 3$,
	let $\bar{\phi}=\frac{a}{b}\bar{\rho}$ and  $bP'(\bar{\rho})-a\mu\bar{\rho}>0$.  If $\|[\rho_0,u_{0},\phi_{0}]\|^2_N+\|\nabla \phi_0]\|^2_N $ is small enough, then the Cauchy problem $\eqref{HPV3}$-$\eqref{NI}$ admits a unique global
	solution $U=[\rho,u, \phi]$ satisfying
	\begin{eqnarray*}
		U \in C([0,\infty);H^{N}(\mathbb{R}^{3})),\ \ \ \nabla\phi \in C([0,\infty);H^{N}(\mathbb{R}^{3})),
	\end{eqnarray*}
	and
	\begin{equation*}
	\begin{split}
		\|[\rho,u,\phi] \|_{N}^{2}+\|\nabla\phi\|_{N}^2+\int_{0}^{t}\left(\|u\|_N^2+\|\nabla\rho\|_{N-1}^{2}
		+\|\nabla \phi\|_{N}^{2}\right)ds
		 \leq& C\left(\|[\rho_0,u_{0},\phi_{0}]\|^2_N+\|\nabla \phi_0]\|^2_N\right),
		 \end{split}
	\end{equation*}
for any $t\geq 0$.
\end{theorem}

To prove Theorem \ref{pro.2.1}, it suffices to derive the {\it a priori } estimates in the following Lemma \ref{estimate}. Before stating the {\it a priori } estimate,  we define the full instant
energy functional $\mathcal {E}_{N}(U(t))$ and corresponding dissipation rate $\mathcal {D}_{N}(U(t))$
 for $N \geq 3$
by
\begin{eqnarray}\label{de.E}
\begin{aligned}
\mathcal {E}_{N}(U(t))=&\sum_{|l|\leq N}\left\{\left\langle\frac{P'(\rho+\bar{\rho})}{\rho+\bar{\rho}},
|\partial^{l}\rho|^2 \right\rangle-2\mu\langle \partial^{l}\phi,\partial^{l}\rho \rangle +\frac{b\mu}{a}\|\partial^{l}\phi\|^2+\langle(\rho+\bar{\rho}),
|\partial^{l}u|^2 \rangle\right\}\\
&+\frac{\mu D}{a}\|\nabla \phi\|_{N}^2+\kappa\sum_{|l|\leq N-1}\left\{\langle
\partial^{l}u,
\partial^{l}\nabla\rho\rangle+\frac{\mu}{2a}\|\partial^{l}\nabla \phi\|^2\right\}
\end{aligned}
\end{eqnarray}
and
\begin{equation}\label{de.D}
\mathcal {D}_{N}(U(t))=\displaystyle \|u\|_N^2+\|\nabla\rho\|_{N-1}^{2}
+\|\nabla \phi\|_{N}^{2},
\end{equation}
where $0<\kappa \ll 1$ is a constant.

For later use and clarity, we present the following Sobolev inequality for the $L^p$ estimate on products of derivatives of two functions (cf. \cite{DRZ}).

\begin{lemma}
 Let  $\theta=\left(\theta_{1}, \cdots, \theta_{n}\right)$ and  $\eta=\left(\eta_{1}, \cdots, \eta_{n}\right)$  be two multi-indices with
$\left|\theta\right|=k_{1},\left|\eta\right|=k_{2}$  and set $k=k_{1}+k_{2}$. Then, for $1 \leq p, q, r \leq \infty$ with  $1 / p=1 / q+1 / r$, we have
	\begin{eqnarray}\label{sobin}
	\big\|\partial^{\theta} u_{1} \partial^{\eta} u_{2}\big\|_{L^{p}(\mathbb{R}^{n})} \leq C\Big(\left\|u_{1}\right\|_{L^{q}\left(\mathbb{R}^{n}\right)}\big\|\nabla^{k} u_{2}\big\|_{L^{r}\left(\mathbb{R}^{n}\right)}+\left\|u_{2}\right\|_{L^{q}\left(\mathbb{R}^{n}\right)}\big\|\nabla^{k} u_{1}\big\|_{L^{r}\left(\mathbb{R}^{n}\right)}\Big),
	\end{eqnarray}
where $C$ is a positive constant.
\end{lemma}

Now, we have the following key {\it a priori } estimates for the solutions to the Cauchy problem \eqref{HPV3}-\eqref{NI}.

\begin{lemma}[{\it A priori } estimates]\label{estimate}
	Assume that the conditions on $\bar{\rho},\bar{\phi}$ in Theorem \ref{pro.2.1} hold. \ Let $U=[\rho,u,\phi]\in
	C([0,T);H^{N}(\mathbb{R}^{3}))$ be a solution of the system (\ref{HPV3}) with  $\|[\rho,u,\phi] \|_{N}^{2}+\|\nabla\phi\|_{N}^2 \ll 1$ for any $0\leq t<T$.
	Then, there exists an energy functional $\mathcal {E}_{N}(\cdot) $  in the form $\eqref{de.E}$  with an dissipation rate $\mathcal {D}_{N}(\cdot) $  in the form $\eqref{de.D} $ such
	that
	\begin{eqnarray}\label{3.2}
	&& \frac{d}{dt}\mathcal {E}_{N}(U(t))+\lambda\mathcal
	{D}_{N}(U(t))\leq 0
	\end{eqnarray}
holds for any $0\leq t<T$ and $0<\lambda<1$.
\end{lemma}

\begin{proof}Our proof is motivated by the work \cite{Kato} and consists of three steps.
	\medskip
	
		\noindent \textbf{ Step 1.} We first claim that
	\begin{equation}\label{3.3}
	\begin{aligned}
	&\frac{1}{2}\frac{d}{dt}\sum_{|l|\leq N}\left\{\left\langle\frac{P'(\rho+\bar{\rho})}{\rho+\bar{\rho}},
	|\partial^{l}\rho|^2 \right\rangle+\langle(\rho+\bar{\rho}),
	|\partial^{l}u|^2 \rangle-2\mu\langle \partial^{l}\phi,\partial^{l}\rho \rangle +\frac{b\mu}{a}\|\partial^{l}\phi\|^2+\frac{\mu D}{a}\|\partial^{l}\nabla \phi\|^2\right\}\\
	&+\alpha\sum_{|l|\leq N}\langle (\rho+\bar{\rho}),
	|\partial^{l}u|^2 \rangle+\frac{\mu}{a}\|\phi_{t}\|_{N}^2\leq  C\|[\rho,u,\phi]\|_{N}\left(\|\nabla[\rho,u,\phi]\|_{N-1}^{2}+\|\nabla\phi\|^2_{N}\right).
	\end{aligned}
	\end{equation}
	In fact, it is convenient to start from the
	following reformulated form of \eqref{HPV3}:
	\begin{equation}\label{HPV3R}
	\left\{\begin{aligned}
	& \pa_{t}\rho+(\rho+\bar{\rho})\nabla\cdot u=-u\cdot\nabla\rho,\\
	&\pa_t u+\frac{P'(\rho+\bar{\rho})}{\rho+\bar{\rho}}\nabla \rho+\alpha u -\mu\nabla\phi=-u\cdot \nabla u,\\
	&\pa_t \phi-D\Delta \phi-a\rho+b\phi=0.
	\end{aligned}\right.
	\end{equation}

	Applying $\partial^{l}$ to the first equation of (\ref{HPV3R}) for
	$0 \leq |l|\leq N$, multiplying the result by $\frac{P'(\rho+\bar{\rho})}{\rho+\bar{\rho}}\partial^{l}\rho$
	and taking integration in $x$ give
	\begin{eqnarray}\label{4.17}
	&& \begin{aligned}
	&\frac{1}{2}\frac{d}{dt}\left\langle\frac{P'(\rho+\bar{\rho})}{\rho+\bar{\rho}},
	|\partial^{l}\rho|^2 \right\rangle +\left\langle P'(\rho+\bar{\rho})\partial^{l}\nabla\cdot u,
	\partial^{l}\rho \right\rangle\\
	=  &\frac{1}{2}\left\langle\left(\frac{P'(\rho+\bar{\rho})}{\rho+\bar{\rho}}\right)_{t}, |\partial^{l}\rho|^2 \right\rangle
	-\sum_{k< l}C^{k}_{l}\left\langle \partial^{l-k}(\rho+\bar{\rho})
	\partial^{k}\nabla\cdot u,\frac{P'(\rho+\bar{\rho})}{\rho+\bar{\rho}}\partial^{l}\rho \right\rangle\\
	&-\left\langle u\cdot\partial^{l}\nabla\rho,\frac{P'(\rho+\bar{\rho})}{\rho+\bar{\rho}}\partial^{l}\rho \right\rangle
	-\sum_{k< l}C^{k}_{l}\left\langle \partial^{l-k}u\cdot
	\partial^{k}\nabla\rho,\frac{P'(\rho+\bar{\rho})}{\rho+\bar{\rho}}\partial^{l}\rho
	\right\rangle.
	\end{aligned}
	\end{eqnarray}
	Applying $\partial^{l}$ to the second equation of (\ref{HPV3R}) for $0\leq |l|\leq N$, multiplying it by $(\rho+\bar{\rho})\partial^{l}u$, and
	integrating the resulting equation with respect to $x$, we get
	\begin{eqnarray}\label{4.4}
	&& \begin{aligned}
	&\frac{1}{2}\frac{d}{dt}\langle(\rho+\bar{\rho}),
	|\partial^{l}u|^2 \rangle+\left\langle P'(\rho+\bar{\rho})\partial^{l}\nabla\rho,
	\partial^{l}u \right\rangle+\alpha\langle (\rho+\bar{\rho}),
	|\partial^{l}u|^2 \rangle
	\\
	=&\frac{1}{2}\langle (\rho+\bar{\rho})_{t},
	|\partial^{l}u|^2 \rangle
	-\sum_{k<l}C^{k}_{l}\left\langle\partial^{l-k}\left(\frac{P'(\rho+\bar{\rho})}{\rho+\bar{\rho}}\right)
	\partial^{k}\nabla\rho,(\rho+\bar{\rho})\partial^{l}u\right\rangle\\
	&-\left\langle u\cdot\partial^{l}\nabla u,(\rho+\bar{\rho})\partial^{l}u \right\rangle
	-\sum_{k<l}C^{k}_{l}\left\langle \partial^{l-k}u\cdot\partial^{k}\nabla u,(\rho+\bar{\rho})\partial^{l}u \right\rangle\\
&+\mu\langle \partial^{l}\nabla\phi,(\rho+\bar{\rho})\partial^{l}u \rangle.
	\end{aligned}
	\end{eqnarray}
With integration by parts and $\eqref{HPV3R}_{1}$, we can reformulate the last term in \eqref{4.4} as follows:
	\begin{eqnarray}\label{phi12}
		\begin{aligned}
			&\mu\langle \partial^{l}\nabla\phi,(\rho+\bar{\rho})\partial^{l}u \rangle\\
			=&-\mu\langle \partial^{l}\phi,(\rho+\bar{\rho})\partial^{l}\nabla\cdot u \rangle-\mu\langle \partial^{l}\phi,\nabla \rho \partial^{l}u \rangle\\
			=&\mu\langle \partial^{l}\phi,\partial^{l}\rho_{t} \rangle+\mu\langle \partial^{l}\phi,\partial^{l}(u \cdot \nabla\rho)\rangle+\mu\left\langle \partial^{l}\phi,\sum_{k<l}C^{k}_{l} \partial^{l-k}(\rho+\bar{\rho})\partial^{k}\nabla\cdot u\right\rangle-\mu\langle \partial^{l}\phi,\nabla \rho \partial^{l}u \rangle\\
			=&\frac{d}{dt}\mu\langle \partial^{l}\phi,\partial^{l}\rho \rangle-\mu\langle \partial^{l}\phi_{t},\partial^{l}\rho \rangle+\mu\langle \partial^{l}\phi,u \cdot\partial^{l} \nabla\rho\rangle-\mu\langle \partial^{l}\phi,\nabla \rho \partial^{l}u \rangle\\
			&+\mu\left\langle \partial^{l}\phi,\sum_{k<l}C^{k}_{l} \partial^{l-k}(\rho+\bar{\rho})\partial^{k}\nabla\cdot u\right\rangle+\mu\left\langle \partial^{l}\phi,\sum_{k<l}C^{k}_{l}\partial^{l-k}u \cdot\partial^{k} \nabla\rho\right\rangle.
		\end{aligned}
	\end{eqnarray}
	
	Applying $\partial^{l}$ to the third equation of (\ref{HPV3R}) for $0\leq |l|\leq N$, multiplying it by $\frac{\mu}{a}\partial^{l}\phi_{t}$, and
	integrating the resultant equation with respect to $x$, we have
	\begin{equation}\label{phit}
	\begin{aligned}
	&\frac{d}{dt}\left\{\frac{b\mu}{2a}\|\partial^{l}\phi\|^2+\frac{\mu D}{2a}\|\partial^{l}\nabla \phi\|^2\right\}+\frac{\mu}{a}\|\partial^{l}\phi_{t}\|^2=
	\mu\langle\partial^{l}\phi_{t} ,
	\partial^{l}\rho \rangle.
	\end{aligned}
	\end{equation}
	It follows from \eqref{4.17}-\eqref{phit} that
	\begin{equation}\label{sum.2417}
	\begin{aligned}
	&\frac{d}{dt}\left(\frac{1}{2}\left\langle\frac{P'(\rho+\bar{\rho})}{\rho+\bar{\rho}},
	|\partial^{l}\rho|^2 \right\rangle+\frac{1}{2}\langle(\rho+\bar{\rho}),
	|\partial^{l}u|^2 \rangle-\mu\langle \partial^{l}\phi,\partial^{l}\rho \rangle +\frac{b\mu}{2a}\|\partial^{l}\phi\|^2+\frac{\mu D}{2a}\|\partial^{l}\nabla \phi\|^2\right)\\
	&+\alpha\langle (\rho+\bar{\rho}),
	|\partial^{l}u|^2 \rangle+\frac{\mu}{a}\|\partial^{l}\phi_{t}\|^2
	=I_{1}^{l}(t)+\sum_{k<l}C^{k}_{l}I_{k,l}^{l}(t),
	\end{aligned}
	\end{equation}
where
	\begin{eqnarray*}
		&& \begin{aligned}
			I_{1}^{l}(t)=&\left\langle P''(\rho+\bar{\rho})\nabla\rho\partial^{l}\rho,
			\partial^{l}u\right\rangle+\frac{1}{2}\left\langle\left(\frac{P'(\rho+\bar{\rho})}{\rho+\bar{\rho}}\right)_{t}, |\partial^{l}\rho|^2 \right\rangle+
			\frac{1}{2}\langle (\rho+\bar{\rho})_{t},
			|\partial^{l}u|^2 \rangle\\
			&+\frac{1}{2}\left\langle\nabla\cdot\left(u\frac{P'(\rho+\bar{\rho})}{\rho+\bar{\rho}}\right),|\partial^{l}\rho|^{2}
			\right\rangle +\frac{1}{2}\left\langle
			\nabla\cdot(u(\rho+\bar{\rho})),|\partial^{l}u|^{2}
			\right\rangle\\
			&-\mu\langle \partial^{l} \nabla \phi,u \partial^{l}\rho\rangle-\mu\langle \partial^{l} \phi,\nabla \cdot u \partial^{l}\rho\rangle-\mu\langle \partial^{l}\phi,\nabla \rho \partial^{l}u \rangle
		\end{aligned}
	\end{eqnarray*}
	and
	\begin{eqnarray*}
		&& \begin{aligned}
			I_{k,l}^{l}(t)=&-\left\langle \partial^{l-k}(\rho+\bar{\rho})
			\partial^{k}\nabla\cdot u,\frac{P'(\rho+\bar{\rho})}{\rho+\bar{\rho}}\partial^{l}\rho \right\rangle-\left\langle \partial^{l-k}u\cdot
			\partial^{k}\nabla\rho,\frac{P'(\rho+\bar{\rho})}{\rho+\bar{\rho}}\partial^{l}\rho
			\right\rangle\\
			&-\left\langle\partial^{l-k}\left(\frac{P'(\rho+\bar{\rho})}{\rho+\bar{\rho}}\right)
			\partial^{k}\nabla\rho,(\rho+\bar{\rho})\partial^{l}u\right\rangle
			-\left\langle \partial^{l-k}u\cdot\partial^{k}\nabla u,(\rho+\bar{\rho})\partial^{l}u \right\rangle\\
			&+\mu\left\langle \partial^{l}\phi,\partial^{l-k}(\rho+\bar{\rho})\partial^{k}\nabla\cdot u\right\rangle+\mu\left\langle \partial^{l}\phi,\partial^{l-k}u \cdot\partial^{k} \nabla\rho\right\rangle.
		\end{aligned}
	\end{eqnarray*}
	
	When $|l|=0$, it suffices to estimate $I^{l}_{1}(t)$ by the Cauchy-Schwarz and Gagliardo-Nirenberg inequalities that
	\begin{eqnarray}\label{Il10}
		\begin{aligned}
			|I_{1}^{l}(t)|
			\leq & C\|\nabla  \rho\|\|\rho\|_{L^{6}}\|u\|_{L^{3}}+ C \|\nabla \cdot u\|(\|\rho\|_{L^{6}}\|\rho\|_{L^{3}}+ \|u\|_{L^{6}}\|u\|_{L^{3}})\\
			&+C\|u\|_{L^{\infty}}(\|\nabla
			\rho\|\|\rho\|_{L^{6}}\|\rho\|_{L^{3}}+\|\nabla
			\rho\|\|u\|_{L^{6}}\|u\|_{L^{3}})+C\|\nabla\phi\|_{L^{2}}\|\rho\|_{L^{6}}\|u\|_{L^{3}}\\
			&+C\|\phi\|_{L^{3}}\|\nabla\cdot u\|_{L^{2}}\|\rho\|_{L^{6}}+C\|\nabla\rho\|_{L^{2}}\|\phi\|_{L^{6}}\|u\|_{L^{3}}\\
			\leq &
			C(\|[\rho,u,\phi]\|_{1}+\|[\rho,u]\|_{2}^{2})\|\nabla[\rho,u,\phi]\|^{2}\\
		\le	&
			C\|[\rho,u,\phi]\|_{2}\|\nabla[\rho,u,\phi]\|^{2},
		\end{aligned}
 	\end{eqnarray}
 where we have used the facts $\|[\rho,u,\phi]\|_{N} \ll 1$ and   $$\pa_{t}\rho+(\rho+\bar{\rho})\nabla\cdot u=-u\cdot\nabla\rho.$$
When $|l|\geq 1$, it follows from the Cauchy-Schwarz inequality that
	\begin{eqnarray*}
\begin{aligned}
|I_{1}^{l}(t)|
\leq & C\|\nabla  \rho\|_{L^\infty}\|\partial^{l}\rho\|\|\partial^{l}u\|+ C\|\nabla \cdot u\|_{L^\infty}(\|\partial^{l}\rho\|^2+\|\partial^{l}u\|^2)\\
&+C\|u\|_{L^{\infty}}\|\nabla \rho\|_{L^\infty}(\|\partial^{l}\rho\|^2+\|\partial^{l}u\|^2)+C\|u\|_{L^{\infty}}(\|\partial^{l}\rho\|^2+\|\partial^{l}\nabla\phi\|^2)\\
&+C\|\nabla\cdot u\|_{L^{\infty}}(\|\partial^{l}\rho\|^2+\|\partial^{l}\phi\|^2)+C\|\nabla \rho\|_{L^{\infty}}(\|\partial^{l}u\|^2+\|\partial^{l}\phi\|^2)\\
\leq &
C\|[\rho,u]\|_{3}\left(\|\nabla
[\rho,u,\phi]\|_{N-1}^{2}+\|\nabla\phi\|_{N}^{2}\right),
\end{aligned}
\end{eqnarray*}	
which along with \eqref{Il10} yields 	
	\begin{eqnarray}\label{Ikl}
			|I_{1}^{l}(t)| \leq C\|[\rho,u]\|_{3}\left(\|\nabla
		[\rho,u,\phi]\|_{N-1}^{2}+\|\nabla\phi\|_{N}^{2}\right).
	\end{eqnarray}
	
By the Cauchy-Schwarz inequality, we have
	\begin{eqnarray}\label{jk}
	\begin{aligned}
		|I_{k,l}^{l}(t)|
		\leq & C\|\partial^{l}\rho\|\|\partial^{l-k}(\rho+\bar{\rho})
		\partial^{k}\nabla\cdot u\|+
		C\|\partial^{l}\rho\|\|\partial^{l-k}u\cdot \partial^{k}\nabla\rho\|\\
		&+C\|\partial^{l}u\|\left\|\partial^{l-k}\left(\frac{P'(\rho+\bar{\rho})}{\rho+\bar{\rho}}\right)
		\partial^{k}\nabla\rho\right\|+C\|\partial^{l}u\|\left\|\partial^{l-k}u\cdot\partial^{k}\nabla u\right\|\\
		&+C\|\partial^{l}\phi\|\left\|\partial^{l-k}(\rho+\bar{\rho})\partial^{k}\nabla\cdot u\right\|+C\|\partial^{l}\phi\|\left\|\partial^{l-k}u \cdot\partial^{k} \nabla\rho\right\|\\
		=& \D\sum_{i=1}^{6}J_i.
	\end{aligned}
\end{eqnarray}	
For $J_1$, noticing that $|l-k|\geq 1$, there exists some multiple index $s$ with $|s|=1$. We have from \eqref{sobin} that
\begin{eqnarray}\label{j1}
	\begin{aligned}
		|J_1|
		\leq & C\|\partial^{l}\rho\|\|\partial^{l-k-s}\partial^{s} \rho
		\partial^{k}\nabla\cdot u\|\\
	\leq &	C\|\partial^{l}\rho\|\left(\|\partial^{s} \rho\|_{L^\infty}
		\|\nabla^{|l|-1}\nabla\cdot u\|+\|\nabla \cdot u\|_{L^\infty}
		\|\nabla^{|l|-1}\partial^{s}\rho\|\right)\\
	\leq & C\|\nabla^2[\rho,u]\|_{1}\|\nabla
	[\rho,u]\|_{N-1}^{2}\\
		\leq & C\|[\rho,u]\|_{N}\|\nabla
	[\rho,u]\|_{N-1}^{2}.\\
	\end{aligned}
\end{eqnarray}	
Similarly, we can estimate $J_2-J_6$ that
\begin{eqnarray}\label{ji}
	\begin{aligned}
		\D\sum_{i=2}^{6}|J_i|
		\leq & 	C\|[\rho,u,\phi]\|_{N}\|\nabla[\rho,u,\phi]\|_{N-1}^{2}.
	\end{aligned}
\end{eqnarray}	
Plugging \eqref{j1} and \eqref{ji} into \eqref{jk}, we see
	\begin{eqnarray}\label{Ikll}
|I_{k,l}^{l}(t) |\leq C\|[\rho,u,\phi]\|_{N}\|\nabla[\rho,u,\phi]\|_{N-1}^{2}.
\end{eqnarray}

Substituting \eqref{Ikl} and \eqref{Ikll} into \eqref{sum.2417},  one has
\begin{equation}\label{sum.2425}
\begin{aligned}
&\frac{d}{dt}\left(\frac{1}{2}\left\langle\frac{P'(\rho+\bar{\rho})}{\rho+\bar{\rho}},
|\partial^{l}\rho|^2 \right\rangle+\frac{1}{2}\langle(\rho+\bar{\rho}),
|\partial^{l}u|^2 \rangle-\mu\langle \partial^{l}\phi,\partial^{l}\rho \rangle +\frac{b\mu}{2a}\|\partial^{l}\phi\|^2+\frac{\mu D}{2a}\|\partial^{l}\nabla \phi\|^2\right)\\
&+\alpha\langle (\rho+\bar{\rho}),
|\partial^{l}u|^2 \rangle+\frac{\mu}{a}\|\partial^{l}\phi_{t}\|^2
\leq  C\|[\rho,u,\phi]\|_{N}\left(\|\nabla[\rho,u,\phi]\|_{N-1}^{2}+\|\nabla\phi\|^2_{N}\right).
\end{aligned}
\end{equation}
Then (\ref{3.3}) follows by taking the summation of (\ref{sum.2425}) over $|l| \leq N$.

	\medskip
	\noindent{\bf Step 2.}
	We show that
	\begin{eqnarray}\label{step2}
	&&\begin{aligned}
	&\frac{d}{dt}\sum_{|l|\leq N-1}\left\{\langle
	\partial^{l}u,
	\partial^{l}\nabla\rho\rangle+\frac{\mu}{2a}\|\partial^{l}\nabla \phi\|^2\right\}
	+\lambda\|\nabla\rho\|_{N-1}^{2}+\lambda\|\nabla \phi\|_{N}^2\\
	\leq &C\|u\|_{N}^2+C\|[\rho,u]\|_{N}^{2}\|\nabla
	[\rho,u]\|_{N-1}^{2}.
	\end{aligned}
	\end{eqnarray}
	Let $0\leq |l|\leq N-1$. Applying $\partial^{l}$ to
	the second equation of $(\ref{HPV3})$, multiplying it by $\partial^{l}\nabla
	\rho$, taking integration in $x$, using
	integration by parts,
	and replacing $\partial_{t}\rho$ from $(\ref{HPV3})_{1}$, one has
	\begin{equation}\label{ste2ru}
	\begin{aligned}
	&
	\frac{d}{dt}\langle
	\partial^{l}u,
	\partial^{l}\nabla\rho\rangle
	+\frac{P'(\bar{\rho})}{\bar{\rho}} \|\partial^{l}\nabla\rho\|^{2}-\mu\left\langle\partial^{l}\nabla\phi,\partial^{l}\nabla\rho\right\rangle\\
	=&\bar{\rho}\|\nabla\cdot\partial^{l}u\|^{2}-
	\langle\nabla\cdot\partial^{l}u,\partial^{l}g_{1}\rangle -\alpha\langle\partial^{l}u,\nabla\partial^{l}\rho\rangle
	+\langle\partial^{l}g_{2},\nabla\partial^{l}\rho\rangle.
	\end{aligned}
	\end{equation}
	Applying $\partial^{l}\nabla$ to the third equation of  (\ref{HPV3}), then multiplying the resultant equation by $\frac{\mu}{a}\partial^{l}\nabla \phi$, and integrating the result with respect to $x$, we get
	\begin{equation}\label{ste2phi}
	\frac{\mu}{2a}\frac{d}{dt}\|\partial^{l}\nabla \phi\|^2+\frac{\mu D}{a}\|\partial^{l}\nabla^2 \phi\|^2-\mu\left\langle\partial^{l}\nabla\phi,\partial^{l}\nabla\rho\right\rangle+\frac{\mu b}{a}\|\partial^{l}\nabla \phi\|^2=0.
	\end{equation}
	Summing \eqref{ste2ru} and \eqref{ste2phi}, one has
	\begin{equation*}
	\begin{aligned}
	&
	\frac{d}{dt}\left\{\langle
	\partial^{l}u,
	\partial^{l}\nabla\rho\rangle+\frac{\mu}{2a}\|\partial^{l}\nabla \phi\|^2\right\}
	+\frac{P'(\bar{\rho})}{\bar{\rho}} \|\partial^{l}\nabla\rho\|^{2}-2\mu\left\langle\partial^{l}\nabla\phi,\partial^{l}\nabla\rho\right\rangle+\frac{\mu b}{a}\|\partial^{l}\nabla \phi\|^2\\
	&+\frac{\mu D}{a}\|\partial^{l}\nabla^2 \phi\|^2=\bar{\rho}\|\nabla\cdot\partial^{l}u\|^{2}-
	\langle\nabla\cdot\partial^{l}u,\partial^{l}g_{1}\rangle -\alpha\langle\partial^{l}u,\nabla\partial^{l}\rho\rangle
	+\langle\partial^{l}g_{2},\nabla\partial^{l}\rho\rangle.
	\end{aligned}
	\end{equation*}
	Notice that the matrix \eqref{matrix}
	is positive definite. Then there exist two positive constants $C_{1}, C_{2}>0$ such that
	\begin{equation*}
	\frac{P'(\bar{\rho})}{\bar{\rho}} \|\partial^{l}\nabla\rho\|^{2}-2\mu\left\langle\partial^{l}\nabla\phi,\partial^{l}\nabla\rho\right\rangle+\frac{\mu b}{a}\|\partial^{l}\nabla \phi\|^2 \geq C_{1}\left( \|\partial^{l}\nabla\rho\|^{2}+ \|\partial^{l}\nabla\phi\|^{2}\right).
	\end{equation*}
	
	Then, it follows from the Cauchy-Schwarz inequality that
	\begin{equation}\label{ineq.2}
	\begin{aligned}
	&
	\frac{d}{dt}\left\{\langle
	\partial^{l}u,
	\partial^{l}\nabla\rho\rangle+\frac{\mu}{2a}\|\partial^{l}\nabla \phi\|^2\right\}
	+\frac{C_{1}}{2}\|\partial^{l}\nabla\rho\|^{2}+C_{1}\|\partial^{l}\nabla \phi\|^2+\frac{\mu D}{a}\|\partial^{l}\nabla^2 \phi\|^2\\
&\leq C\left(\|\nabla\cdot\partial^{l}u\|^{2}+\|\partial^{l}u\|^2\right)+C\left(\|\partial^{l}g_{1}\|^2
	+\|\partial^{l}g_{2}\|^2\right).
	\end{aligned}
	\end{equation}
	Noticing that  $g_{1},\ g_{2}$ are quadratically
	nonlinear, one has from \eqref{sobin} that
	\begin{equation*}
	\|\partial^{l}g_{1}\|^{2}+\|\partial^{l}g_{2}\|^{2}\leq
	C\|[\rho,u]\|_{N}^{2}\|\nabla
	[\rho,u]\|_{N-1}^{2}.
	\end{equation*}
	Substituting this into (\ref{ineq.2}) and taking
	the summation over $|l|\leq N-1$ imply (\ref{step2}).

	\medskip
	
	\noindent\textbf{Step 3.}
	Multiplying (\ref{step2}) by $\kappa$ and adding resulting inequality to \eqref{3.3}, we have
	\begin{equation*}
	\begin{aligned}
	&\frac{1}{2}\frac{d}{dt}\sum_{|l|\leq N}\left\{\left\langle\frac{P'(\rho+\bar{\rho})}{\rho+\bar{\rho}},
	|\partial^{l}\rho|^2 \right\rangle+\langle(\rho+\bar{\rho}),
	|\partial^{l}u|^2 \rangle-2\mu\langle \partial^{l}\phi,\partial^{l}\rho \rangle +\frac{b\mu}{a}\|\partial^{l}\phi\|^2+\frac{\mu D}{a}\|\partial^{l}\nabla \phi\|^2\right\}\\
	&+\kappa\frac{d}{dt}\sum_{|l|\leq N-1}\left\{\langle
	\partial^{l}u,
	\partial^{l}\nabla\rho\rangle+\frac{\mu}{2a}\|\partial^{l}\nabla \phi\|^2\right\}+\kappa\lambda\|\nabla\rho\|_{N-1}^{2}+\kappa\lambda\|\nabla \phi\|_{N}^2\\
	&+\alpha\sum_{|l|\leq N}\langle (\rho+\bar{\rho}),
	|\partial^{l}u|^2 \rangle+\frac{\mu}{a}\|\phi_{t}\|_{N}^2\\
	\leq& C\kappa\|u\|_{N}^2+ C\|[\rho,u,\phi]\|_{N}\left(\|\nabla[\rho,u,\phi]\|_{N-1}^{2}+\|\nabla\phi\|^2_{N}\right).
	\end{aligned}
	\end{equation*}
	By choosing $\kappa$ and $\|[\rho,u,\phi]\|_{N}$ small enough, we can obtain \eqref{3.2} with \eqref{de.E} and \eqref{de.D}.	This completes the proof of Lemma $\ref{estimate}$.
\end{proof}
\noindent\textbf{Proof of Theorem \ref{pro.2.1}}. We  rewrite $\left\langle\frac{P'(\rho+\bar{\rho})}{\rho+\bar{\rho}},
	|\partial^{l}\rho|^2 \right\rangle$ as follows:
	$$
	\left\langle\frac{P'(\rho+\bar{\rho})}{\rho+\bar{\rho}},
	|\partial^{l}\rho|^2 \right\rangle=\frac{P'(\bar{\rho})}{\bar{\rho}}
	\|\partial^{l}\rho\|^2+\left\langle\frac{P'(\rho+\bar{\rho})}{\rho+\bar{\rho}}-\frac{P'(\bar{\rho})}{\bar{\rho}},
	|\partial^{l}\rho|^2 \right\rangle,
	$$
	which updates \eqref{de.E} to
  \begin{eqnarray}\label{3.12}
	\begin{aligned}
	\mathcal {E}_{N}(U(t))=&\sum_{|l|\leq N}\left\{\frac{P'(\bar{\rho})}{\bar{\rho}}
	\|\partial^{l}\rho\|^2-2\mu\langle \partial^{l}\phi,\partial^{l}\rho \rangle +\frac{b\mu}{a}\|\partial^{l}\phi\|^2+\langle(\rho+\bar{\rho}),
	|\partial^{l}u|^2 \rangle\right\}\\
	&+\sum_{|l|\leq N}\left\langle\frac{P'(\rho+\bar{\rho})}{\rho+\bar{\rho}}-\frac{P'(\bar{\rho})}{\bar{\rho}},
	|\partial^{l}\rho|^2 \right\rangle+\frac{\mu D}{a}\|\nabla \phi\|_{N}^2\\
	&+\kappa\sum_{|l|\leq N-1}\left\{\langle
	\partial^{l}u,
	\partial^{l}\nabla\rho\rangle+\frac{\mu}{2a}\|\partial^{l}\nabla \phi\|^2\right\}
	\end{aligned}
	\end{eqnarray}
	with constant $0<\kappa\ll 1$. By the fact that the matrix \eqref{matrix}
is positive definite, along with the smallness of $\kappa$ and $\|\rho\|_N$, we have that
\begin{equation*}
\mathcal {E}_{N}(U(t))\sim
\|[\rho,u,\phi] \|_{N}^{2}+\|\nabla\phi\|_{N}^2.
\end{equation*}
This together with \eqref{3.2} leads to
\begin{equation*}
\begin{split}
&\|[\rho,u,\phi] \|_{N}^{2}+\|\nabla\phi\|_{N}^2+\int_{0}^{t}\left(\|u\|_N^2+\|\nabla\rho\|_{N-1}^{2}
+\|\nabla \phi\|_{N}^{2}\right)ds\\
\leq& C\left(\|[\rho_0,u_{0},\phi_{0}]\|^2_N+\|\nabla \phi_0]\|^2_N\right).
\end{split}
\end{equation*}
This {\it a priori} estimates combined with the local existence theorem completes the proof of Theorem \ref{pro.2.1}.

\subsection{Asymptotic decay rate to constant states}
In what follows, since we shall apply the linear $L^{p}$-$L^{q}$
time-decay property of the homogeneous system \eqref{HPVL}-\eqref{NLI} to  the nonlinear case, we need
the mild form of the non-linear Cauchy problem
\eqref{HPV3}-\eqref{NI}. By Duhamel's
principle, the solution $U$ can be formally written as
 \begin{eqnarray}\label{sec5.U}
\begin{aligned}
U(t)=&e^{tL}U_{0}+\int_{0}^{t}e^{(t-s)L}[g_{1}(s),g_{2}(s),0]d
s\\
=&e^{tL}U_{0}+\int_{0}^{t}e^{(t-s)L}[\nabla \cdot
f_{1}(s),g_{2}(s),0]d s,
\end{aligned}
\end{eqnarray}
where $e^{tL}$ is the linearized solution operator and the nonlinear source terms $g_{1},g_{2},f_{1}$ were defined in \eqref{sec5.ggg}.
For later use with clarity, we introduce some basic inequalities without proof.
\begin{lemma}\label{halft}For any $r_{1}\in [0,r_{2}]$, $\lambda>0$, and $\alpha>0$, there exists a constant $C>0$ such that
$$
\begin{aligned}
&\int_{0}^{t}(1+t-s)^{-r_{1}}(1+s)^{-r_{2}}ds \leqslant
\arraycolsep=1.5pt
\left\{
  \begin{array}{l}
C(r_{1},r_{2})(1+t)^{-(r_{1}+r_{2}-1)},\ \ \ \ \ \ r_{2}<1\\[3mm]
C(r_{1},r_{2})(1+t)^{-r_{1}}\ln(1+t),\ \ \ \ r_{2}=1 \\[3mm]
C(r_{1},r_{2})(1+t)^{-r_{1}}, \ \ \ \ \ \ \ \ \ \ \ \ \ \ \ r_{2}>1
\end{array}
\right.\\
\text{and} \\
&\int_{0}^{t}{\mathop{\mathrm{e}}}^{-\lambda (t-s)}(1+s)^{-\alpha}ds \leqslant C(1+t)^{-\alpha}.
\end{aligned}
$$
\end{lemma}

Below we shall show that the solutions obtained in Theorem $ \ref{pro.2.1}$ enjoy the algebraic decay rates under some additional regularity and integrability conditions on initial data. To this end, for given
$U_{0}=[\rho_{0},u_{0}, \phi_0]$, we set
$\epsilon_{m}(U_0)$ as
\begin{eqnarray}\label{def.epsi}
\epsilon_{m}(U_0)=\|U_{0}\|_{m}+\|\nabla\phi_{0}\|_{m}+\|U_{0}\|_{L^{1}},
\end{eqnarray}
for the  integer $m \geq 0$. Then one has the following theorem.

\begin{theorem}\label{pro.2.2}
	Under the assumptions of Theorem \ref{pro.2.1}, if
	$\epsilon_{N}(U_{0})>0$ is small enough, then the solution
	$U=[\rho,u, \phi]$ of \eqref{HPV3}-\eqref{NI}  satisfies
	\begin{eqnarray}\label{V.decay}
	\|U(t)\|_{N}+\|\nabla\phi\|_{N} \leq C \epsilon_{N}(U_{0})(1+t)^{-\frac{3}{4}},
	\end{eqnarray}
	and
	\begin{eqnarray}\label{nablaV.decay}
	\|\nabla U(t)\|_{N-1}+\|\nabla^2\phi\|_{N-1} \leq C
	\epsilon_{N}(U_{0})(1+t)^{-\frac{5}{4}},
	\end{eqnarray}
	for any $t\geq 0$.
\end{theorem}

Next, we give the proof of  Theorem \ref{pro.2.2} which consists of several steps shown below.

\subsubsection{Decay rate for the full instant energy functional}
Recall from the proof of Lemma \ref{estimate} that
\begin{eqnarray}\label{sec5.ENV0}
\dfrac{d}{dt}\mathcal {E}_{N}(U(t))+\lambda \mathcal {D}_{N}(U(t))
\leq 0,
\end{eqnarray}
for any $t\geq 0$. We now apply the time-weighted energy estimate and iteration to the
Lyapunov inequality $\eqref{sec5.ENV0}$. Let $\ell \geq 0$. Multiplying
$\eqref{sec5.ENV0}$ by $(1+t)^{\ell}$ and taking integration over
$[0,t]$ give
\begin{eqnarray*}
	\begin{aligned}
		& (1+t)^{\ell}\mathcal {E}_{N}(U(t))+\lambda
		\int_{0}^{t}(1+s)^{\ell}\mathcal {D}_{N}(U(s))d s \\
		\leq & \mathcal {E}_{N}(U_{0})+ \ell
		\int_{0}^{t}(1+s)^{\ell-1}\mathcal {E}_{N}(U(s))d s.
	\end{aligned}
\end{eqnarray*}
Noticing that
\begin{eqnarray*}
	\mathcal {E}_{N}(U(t))
	\leq C (D_{N}(U(t))+ \|[\rho(t),\phi(t)]\|^{2}),
\end{eqnarray*}
we have
\begin{eqnarray*}
	\begin{aligned}
		& (1+t)^{\ell}\mathcal {E}_{N}(U(t))+\lambda
		\int_{0}^{t}(1+s)^{\ell}\mathcal {D}_{N}(U(s))d s \\
		\leq & \mathcal {E}_{N}(U_{0})+ C \ell
		\int_{0}^{t}(1+s)^{\ell-1}\|[\rho(s),\phi(s)]\|^{2} d s\\
		&+ C\ell\int_{0}^{t}(1+s)^{\ell-1}\mathcal {D}_{N}(U(s))d s.
	\end{aligned}
\end{eqnarray*}
Similarly, it holds that
\begin{eqnarray*}
	\begin{aligned}
		& (1+t)^{\ell-1}\mathcal {E}_{N}(U(t))+\lambda
		\int_{0}^{t}(1+s)^{\ell-1}\mathcal {D}_{N}(U(s))d s \\
		\leq & \mathcal {E}_{N}(U_{0})+ C (\ell-1)
		\int_{0}^{t}(1+s)^{\ell-2}\|[\rho(s),\phi(s)]\|^{2}d s\\
		&+ C(\ell-1)\int_{0}^{t}(1+s)^{\ell-2}\mathcal {D}_{N}(U(s))d s,
	\end{aligned}
\end{eqnarray*}
and
\begin{eqnarray*}
	\mathcal {E}_{N}(U(t))+\lambda \int_{0}^{t}\mathcal
	{D}_{N}(U(s))d s \leq \mathcal {E}_{N}(U_{0}).
\end{eqnarray*}
Then, for $1<\ell<2$, it follows by iterating the above estimates
that
\begin{eqnarray}\label{sec5.ED}
\begin{aligned}
& (1+t)^{\ell}\mathcal {E}_{N}(U(t))+\lambda
\int_{0}^{t}(1+s)^{\ell}\mathcal {D}_{N}(U(s))d s \\
\leq & C \mathcal {E}_{N}(U_{0})+ C
\int_{0}^{t}(1+s)^{\ell-1}\|[\rho(s),\phi(s)]\|^{2}d s.
\end{aligned}
\end{eqnarray}
On the other hand, to estimate the integral term on the right-hand side  of
$\eqref{sec5.ED}$, let's define
\begin{eqnarray}\label{sec5.def}
\mathcal {E}_{N,\infty}(U(t))=\sup\limits_{0\leq s \leq t}
(1+s)^{\frac{3}{2}}\mathcal {E}_{N}(U(s)).
\end{eqnarray}

Then we have the following estimates.
\begin{lemma}\label{lem.Bsigma}
	For any $t\geq0$, it holds that
	\begin{eqnarray}\label{lem.tildeB}
	&&\begin{aligned}
	\|[\rho(t),\phi(t)]\|^{2}
	\leq  C(1+t)^{-\frac{3}{2}}\left(\mathcal
	{E}_{N,\infty}^{2}(U(t))+\|U_{0}\|_{L^{1}\cap  L^{2}}^{2}+\|\nabla\phi_{0}\|^2\right).
	\end{aligned}
	\end{eqnarray}
\end{lemma}

\begin{proof}
	By applying the linear estimates on $\rho$ and $\phi$ with   $m=0,\ q=r=2,\ p=s=1 $ in  Proposition \ref{thm.decayrup} to the mild solution form \eqref{sec5.U} respectively, one has
	\begin{eqnarray}\label{sec5.decayrho}
\begin{aligned}
	\|[\rho(t),\phi(t)]\|\leq & C
	(1+t)^{-\frac{3}{4}} \left(\|U_{0}\|_{L^{1}\cap L^{2}}+\|\nabla\phi_{0}\|\right)\\
	&+C
	\int_{0}^{t}(1+t-s)^{-\frac{3}{4}}\left(\|[g_{1}(s),g_{2}(s)]\|_{L^{1}\cap L^{2}}\right)ds.
\end{aligned}
	\end{eqnarray}
	Recall the definition $\eqref{sec5.ggg}$ of $g_{1}$ and
	$g_{2}$. It is straightforward to verify that for
	any $0\leq s\leq t$,
	\begin{eqnarray*}
		\|[g_{1}(s),g_{2}(s)]\|_{L^{1}\cap L^{2}}\leq
		C \mathcal {E}_{N}(U(s))\leq (1+s)^{-\frac{3}{2}}\mathcal
		{E}_{N,\infty}(U(t)),
	\end{eqnarray*}
	where we have used $\eqref{sec5.def}$. Putting the above inequality into $\eqref{sec5.decayrho}$ and using Lemma \ref{halft} give
	\begin{eqnarray*}
		\|[\rho,\phi](t)\|\leq C (1+t)^{-\frac{3}{4}}
		(\|U_{0}\|_{L^{1}\cap L^{2}}+\|\nabla\phi_{0}\|_{L^{2}}+\mathcal {E}_{N,\infty}(U(t))),
	\end{eqnarray*}
	which implies $\eqref{lem.tildeB}$.  This completes the proof of Lemma
	$\ref{lem.Bsigma}$.
\end{proof}
Next, we prove the uniform-in-time bound of $\mathcal
{E}_{N,\infty}(U(t))$ which yields the time-decay rates of the
Lyapunov functional $\mathcal
{E}_{N}(U(t))$ and thus $\|U(t)\|_{N}^{2}$. In fact, by taking $\ell =\frac{3}{2}+\epsilon$
in $\eqref{sec5.ED}$ with $\epsilon>0$ small enough, one has
\begin{eqnarray*}
	\begin{aligned}
		& (1+t)^{\frac{3}{2}+\epsilon}\mathcal {E}_{N}(U(t))+\lambda
		\int_{0}^{t}(1+s)^{\frac{3}{2}+\epsilon}\mathcal {D}_{N}(U(s))d s \\
		\leq & C \mathcal {E}_{N}(U_{0})+ C
		\int_{0}^{t}(1+s)^{\frac{1}{2}+\epsilon}\|[\rho(s),\phi(s)]\|^{2}d s.
	\end{aligned}
\end{eqnarray*}
Here, using $\eqref{lem.tildeB}$ and the fact that $\mathcal
{E}_{N,\infty}(U(t))$ is non-decreasing in $t$, we have that
\begin{eqnarray*}
	\begin{aligned}
		&\int_{0}^{t}(1+s)^{\frac{1}{2}+\epsilon}\|[\rho(s),\phi(s)]\|^{2}d s\\
		\leq  & C(1+t)^{\epsilon}\left(\mathcal
		{E}_{N,\infty}^{2}(U(t))+\|U_{0}\|_{L^{1}\cap L^{2}}^{2}+\|\nabla\phi_{0}\|^2\right).
	\end{aligned}
\end{eqnarray*}
Therefore, it follows that
\begin{eqnarray*}
	\begin{aligned}
		& (1+t)^{\frac{3}{2}+\epsilon}\mathcal {E}_{N}(U(t))+\lambda
		\int_{0}^{t}(1+s)^{\frac{3}{2}+\epsilon}\mathcal {D}_{N}(U(s))d s \\
		\leq & C \mathcal {E}_{N}(U_{0})+ C(1+t)^{\epsilon}\left(\mathcal
		{E}_{N,\infty}^{2}(U(t))+\|U_{0}\|_{L^{1}\cap L^{2}}^{2}+\|\nabla\phi_{0}\|^2\right),
	\end{aligned}
\end{eqnarray*}
which implies
\begin{eqnarray*}
	\begin{aligned}
		(1+t)^{\frac{3}{2}}\mathcal {E}_{N}(U(t))
		\leq & C\left( \mathcal {E}_{N}(U_{0})+\mathcal
		{E}_{N,\infty}^{2}(U(t))+\|U_{0}\|_{L^{1}\cap L^{2}}^{2}+\|\nabla\phi_{0}\|^2\right) .
	\end{aligned}
\end{eqnarray*}
Thus, one has
\begin{eqnarray*}
	\mathcal {E}_{N,\infty}(U(t))
	\leq C \left( \epsilon_{N}^{2}(U_{0})+
	\mathcal {E}_{N,\infty}^{2}(U(t))\right).
\end{eqnarray*}
Here,  we have used the definition of $\epsilon_{N}(U_{0})$. Since
$\epsilon_{N}(U_{0})>0$ is sufficiently small, $\mathcal
{E}_{N,\infty}(U(t)) \leq C \epsilon_{N}^{2}(U_{0})$ holds true
for any $t\geq 0$, which implies
\begin{eqnarray*}
\|U(t)\|_{N}+\|\nabla\phi\|_{N}\leq C \mathcal {E}_{N}(U(t))^{1/2}
\leq C  \epsilon_{N}(U_{0})(1+t)^{-\frac{3}{4}}
\end{eqnarray*}
for any $t\geq 0$. This yields \eqref{V.decay} in Theorem
\ref{pro.2.2}.

\subsubsection{Decay rate for the higher-order  instant
	energy functional}
	In this subsection, we shall continue the proof of Theorem \ref{pro.2.2} for the second part \eqref{nablaV.decay}. In fact, it can be
reduced to the time-decay estimates only on $\|\nabla [\rho,\phi]\|$ by the following lemma.

\begin{lemma}\label{estimate2}
	Let $U=[\rho,u,\phi]$ be the solution to the Cauchy
	problem $\eqref{HPV3}$-$ \eqref{NI}$. Then if
	$\mathcal {E}_{N}(U_{0})$ is sufficiently small, there exist a
	high-order instant energy functional $\mathcal {E}_{N}^{h}(\cdot)$ with
	\begin{equation*}
	\mathcal
	{E}_{N}^{h}(U(t))\sim \|\nabla U(t)\|_{N-1}^{2}+\|\nabla^2\phi\|_{N-1},
	\end{equation*}
	and the corresponding dissipation rate $\mathcal {D}_{N}^{h}(\cdot)$ satisfying
	\begin{eqnarray}\label{sec5.high}
	&& \frac{d}{dt}\mathcal {E}_{N}^{h}(U(t))+\lambda\mathcal
	{D}^{h}_{N}(U(t))\leq C\|\nabla[
	\rho,\phi]\|^{2}
	\end{eqnarray}
for any $ t \geq 0$.
\end{lemma}

\begin{proof}
	The proof can be carried out by slightly modifying the proof of Theorem
	$\ref{estimate}$. In fact, by conducting the energy estimates on the only  high-order derivatives, similar to
	$\eqref{3.3}$ and $\eqref{step2}$,
	it can be verified that
	
	\begin{eqnarray*}
		&&\begin{aligned}
			&\frac{1}{2}\frac{d}{dt}\sum_{1\leq |l|\leq N}\left\{\left\langle\frac{P'(\rho+\bar{\rho})}{\rho+\bar{\rho}},
			|\partial^{l}\rho|^2 \right\rangle+\langle(\rho+\bar{\rho}),
			|\partial^{l}u|^2 \rangle-2\mu\langle \partial^{l}\phi,\partial^{l}\rho \rangle +\frac{b\mu}{a}\|\partial^{l}\phi\|^2+\frac{\mu D}{a}\|\partial^{l}\nabla \phi\|^2\right\}\\
			&+\alpha\sum_{1\leq |l|\leq N}\langle (\rho+\bar{\rho}),
			|\partial^{l}u|^2 \rangle+\frac{\mu}{a}\|\nabla\phi_{t}\|_{N-1}^2\leq  C\|[\rho,u,\phi]\|_{N}\left(\|\nabla[\rho,u,\phi]\|_{N-1}^{2}+\|\nabla^2\phi\|^2_{N-1}\right)
		\end{aligned}
	\end{eqnarray*}
	and
	\begin{eqnarray*}
		&&\begin{aligned}
			&\frac{d}{dt}\sum_{1\leq |l|\leq N-1}\left\{\langle
			\partial^{l}u,
			\partial^{l}\nabla\rho\rangle+\frac{\mu}{2a}\|\partial^{l}\nabla \phi\|^2\right\}
			+\lambda\|\nabla^2\rho\|_{N-2}^{2}+\lambda\|\nabla^2\phi\|_{N-1}^2\\
			\leq &C\|\nabla u\|_{N-1}^2+C\|[\rho,u]\|_{N}^{2}\|\nabla
			[\rho,u]\|_{N-1}^{2}.
		\end{aligned}
	\end{eqnarray*}
	Here, the details of proof are omitted for simplicity. Now, similar
	to $\eqref{3.12}$, let us define
	
	\begin{eqnarray}\label{de.Eh}
	\begin{aligned}
	\mathcal {E}^{h}_{N}(U(t))=&\sum_{1\leq |l|\leq N}\left\{\left\langle\frac{P'(\rho+\bar{\rho})}{\rho+\bar{\rho}},
	|\partial^{l}\rho|^2 \right\rangle-2\mu\langle \partial^{l}\phi,\partial^{l}\rho \rangle +\frac{b\mu}{a}\|\partial^{l}\phi\|^2+\langle(\rho+\bar{\rho}),
	|\partial^{l}u|^2 \rangle\right\}\\
	&+\frac{\mu D}{a}\|\nabla^2\phi\|_{N-1}^2+\kappa\sum_{1\leq |l|\leq N-1}\left\{\langle
	\partial^{l}u,
	\partial^{l}\nabla\rho\rangle+\frac{\mu}{2a}\|\partial^{l}\nabla \phi\|^2\right\},
	\end{aligned}
	\end{eqnarray}
	and
\begin{equation}\label{de.Dh}
\mathcal {D}^h_{N}(U(t))=\displaystyle\alpha \sum_{1\leq|l|\leq
	N}\int_{\mathbb{R}^3}(\rho+\bar{\rho})
|\partial^{l}u|^2dx+\lambda\|\nabla^2\rho\|_{N-2}^{2}
+\lambda\|\nabla ^2\phi\|_{N-1}^{2}.
\end{equation}
Thus, $\mathcal
	{E}_{N}^{h}(U(t))\sim \|\nabla U(t)\|_{N-1}^{2}+\|\nabla^2\phi\|_{N-1}$ for $0<\kappa \ll 1$. Furthermore, the
	linear combination of previously obtained two estimates with
	coefficients corresponding to $\eqref{de.Eh}$ yields
	$\eqref{sec5.high}$. This completes the proof of Lemma \ref{estimate2}.
\end{proof}

By comparing \eqref{de.Dh} with $\eqref{de.Eh}$ for the definitions
of $ \mathcal {E}_{N}^{h}(U(t))$ and  $ \mathcal {D}_{N}^{h}(U(t))$,
it follows from $\eqref{sec5.high}$ that
\begin{eqnarray*}
	&& \frac{d}{dt}\mathcal {E}_{N}^{h}(U(t))+\lambda\mathcal
	{E}^{h}_{N}(U(t))\leq
	C\|\nabla [\rho,\phi](t)\|^{2},
\end{eqnarray*}
which implies
\begin{eqnarray}\label{sec.ee}
\mathcal {E}_{N}^{h}(U(t))\leq {e^{-\lambda t}}
\mathcal{E}_{N}^{h}(U_{0}) +C\int_{0}^{t}{\exp\{-\lambda(t-s)\}}\|\nabla [\rho,\phi](s)\|^{2}d s.
\end{eqnarray}
To estimate the time integral term on the right-hand side of the above
inequality, we show the following result.

\begin{lemma}\label{estimate3}
	Under the assumptions of Theorem $\ref{pro.2.1}$, if $
	\epsilon_{N}(U_{0})$ defined in $\eqref{def.epsi}$ is sufficiently small then
	\begin{eqnarray}\label{sec5.highBE}
	&&\begin{aligned} \|\nabla [\rho,\phi](t)\|^{2}\leq
	C\epsilon_{N}^{2}(U_{0})(1+t)^{-\frac{5}{2}}
	\end{aligned}
	\end{eqnarray}
	holds for any $ t \geq 0$.
\end{lemma}
\begin{proof} Suppose that
	$\epsilon_{N}(U_{0})>0$ is sufficiently small. It follows from \eqref{V.decay} that
	\begin{eqnarray}\label{UN+4}
	\|U(t)\|_{N}+\|\nabla\phi(t)\|_{N}  \leq C  \epsilon_{N}(U_{0})(1+t)^{-\frac{3}{4}}.
	\end{eqnarray}
	Applying the linear
	estimate on $\rho,\phi$  by setting $m=1,\ q=r=s=2,\ p=1$ in Proposition \ref{thm.decayrup} to the mild
	form \eqref{sec5.U} respectively, one has
	\begin{eqnarray}\label{sec5.nablarho}
\begin{aligned}
	\|\nabla [\rho,\phi] (t)\|\leq & C (1+t)^{-\frac{5}{4}}\left(
	\|U_{0}\|_{ L^1\cap\dot{H}^{1}}+\|\nabla^{2}\phi_{0}\|\right)\\
   &+C	\int_{0}^{t}(1+t-s)^{-\frac{5}{4}}\|[g_{1}(s),g_{2}(s),0]\|_{L^{1}\cap\dot{H}^{1}}ds.
   \end{aligned}
	\end{eqnarray}
	Recalling the definition $\eqref{sec5.ggg}$, we can verify that
	\begin{eqnarray*}
		\|[g_{1}(s),g_{2}(s)(s)]\|_{L^{1}\cap\dot{H}^{1}}\leq
		C\|U(t)\|_{3}^{2}\leq C
		\epsilon_{N}^{2}(U_{0})(1+s)^{-\frac{3}{2}}.
	\end{eqnarray*}
	Then it follows from $\eqref{UN+4}$, $\eqref{sec5.nablarho}$ and Lemma \ref{halft} that
	\begin{eqnarray*}
		\begin{aligned}
			\|\nabla [\rho,\phi](t)\| \leq
			C\epsilon_{N}(U_{0})(1+t)^{-\frac{5}{4}}.
		\end{aligned}
	\end{eqnarray*}
The proof  of Lemma $\ref{estimate3}$ is completed.
\end{proof}
By $\eqref{sec.ee}$, $\eqref{sec5.highBE}$  and Lemma \ref{halft}, we immediately have
\begin{eqnarray*}
	\mathcal {E}_{N}^{h}(U(t))\leq {\exp\{-\lambda t\}}
	\mathcal{E}_{N}^{h}(U_{0})+C
	\epsilon_{N}^{2}(U_{0})(1+t)^{-\frac{5}{2}},
\end{eqnarray*}
which proves  \eqref{nablaV.decay} in Theorem \ref{pro.2.2}.
\medskip

Based on the decay properties for $\|U\|_{N}$ and $\|\nabla U\|_{N-1}$ in Theorem \ref{pro.2.2}, we can obtain the decay estimates of  $\|\nabla u\|$ and $\|\nabla^2\rho\|$, which will be used later to explore the faster decay rates from nonlinear equation to the linear case.
\begin{lemma}
	Under the assumptions of Theorem $\ref{pro.2.1}$, if $
	\epsilon_{N}(U_{0})$ defined in $\eqref{def.epsi}$ is sufficiently small then
	\begin{eqnarray}\label{sec5.highu1}
	&&\begin{aligned} \|\nabla u(t)\|\leq
	C\epsilon_{N}(U_{0})(1+t)^{-\frac{7}{4}},
	\end{aligned}
	\end{eqnarray}
	and
	\begin{eqnarray}\label{sec5.highrho}
	&&\begin{aligned} \|\nabla^2 \rho(t)\|\leq
	C\epsilon_{N}(U_{0})(1+t)^{-\frac{7}{4}}
	\end{aligned}
	\end{eqnarray}
	hold for any $ t \geq 0$.
\end{lemma}
\begin{proof}
By applying the linear estimates on $u$ with  $m=1,\ q=r=s=2,\ p=1 $ in  Proposition \ref{thm.decayrup} to the mild solution form \eqref{sec5.U}, one has
\begin{eqnarray}\label{nablau}
		\begin{aligned}
	\|\nabla u(t)\|\leq & C(1+t)^{-\frac{7}{4}}\|\rho_{0}\|_{L^{1}}
	+C(1+t)^{-\frac{9}{4}}\left(\|U_{0}\|_{L^{1}\cap \dot{H}^{1}}+\|\nabla^2 [\rho_{0},\phi_{0}]\|\right)\\
	&+C\int_{0}^{t}\big((1+t-s)^{-\frac{7}{4}}\|g_{1}(s)\|_{L^{1}}\big)dx\\
&+\int_{0}^{t}(1+t-s)^{-\frac{9}{4}}\left(\|[g_{1},g_{2}](s)\|_{L^{1}\cap \dot{H}^{1}}\|\nabla^2 g_{1}(s)\|\right)ds.
	\end{aligned}
	\end{eqnarray}
By applying the linear estimates on $\rho$ with  $m=2,\ q=r=s=2,\ p=1 $ in  Proposition \ref{thm.decayrup} to the mild solution form \eqref{sec5.U}, we obtain
\begin{eqnarray}\label{nabla2rho}
		\begin{aligned}
	\|\nabla^2\rho(t)\| \leq &C(1+t)^{-\frac{7}{4}}\|\rho_{0}\|_{L^{1}}
	+C(1+t)^{-\frac{9}{4}}\left(\|U_{0}\|_{L^{1}\cap \dot{H}^{2}}+\|\nabla^2 \phi_{0}\|\right)\\
	&+C\int_{0}^{t}(1+t-s)^{-\frac{7}{4}}\|g_{1}(s)\|_{L^{1}}+(1+t-s)^{-\frac{9}{4}}\|[g_{1},g_{2}](s)\|_{L^{1}\cap \dot{H}^{2}}ds.
	\end{aligned}
	\end{eqnarray}
Recall the definition $\eqref{sec5.ggg}$, it is straightforward to verify
	\begin{eqnarray*}
		\|[g_{1},g_{2}](s)\|_{L^{1}\cap H^{2}}\leq
		C\|U(t)\|_{3}\|\nabla U(t)\|_{2}+\|\nabla U(t)\|^2_{2}\leq C
		\epsilon_{N}^{2}(U_{0})(1+s)^{-2},
	\end{eqnarray*}
	and
	\begin{eqnarray*}
		\|\nabla^2g_{1}(s)\|_{L^{2}}\leq
		C\|\nabla U(t)\|^2_{2}\leq C
		\epsilon_{N}^{2}(U_{0})(1+s)^{-\frac{5}{2}}.
	\end{eqnarray*}
	Then it follows from \eqref{nablau}, \eqref{nabla2rho} and Lemma \ref{halft} that
	\begin{eqnarray*}
		\begin{aligned}
			\|\nabla u(t)\|+\|\nabla^2\rho\| \leq
			C\epsilon_{N}(U_{0})(1+t)^{-\frac{7}{4}}.
		\end{aligned}
	\end{eqnarray*}
	\end{proof}
\subsubsection{Decay rate of $L^{q}$-norm}
For $L^{2}$ rate of $\rho$ and $\phi$, it is
easy to see from Lemma $\ref{lem.Bsigma}$ that
\begin{eqnarray}
\|\rho(t)\|+\|\phi (t)\| \leq C
\epsilon_{3}(U_{0})(1+t)^{-\frac{3}{4}}.\label{rate.brho}
\end{eqnarray}
Applying the $L^{\infty}$ linear estimate
on $[\rho,\phi]$ with  $m=0,\ q=\infty, r=2,\ s=p=1 $ in Proposition $\ref{thm.decayrup}$ to the mild form \eqref{sec5.U},
one has
\begin{eqnarray*}
\begin{aligned} \|[\rho,\phi](t)\|_{L^{\infty}}\leq & C (1+t)^{-\frac{3}{2}}
 \|\rho_{0}\|_{L^1}+ (1+t)^{-2}
 \|U_{0}\|_{L^1}+e^{-\lambda t}\|U_{0}\|_{L^{1}}\\
 &+e^{-\lambda t}\|\nabla^2 U_{0}\|_{L^{2}}+e^{-\lambda t}\|\nabla ^3\phi_{0}\|_{L^{2}}\\
&+C
\int_{0}^{t}\left\{(1+t-s)^{-\frac{3}{2}} \|g_{1}(s)\|_{L^1}+ (1+t-s)^{-2}
 \|[g_{1},g_{2}](s)\|_{L^1}\right\}ds\\
 &+C
\int_{0}^{t}\left\{e^{-\lambda (t-s)}\|[g_{1},g_{2}](s)\|_{L^{1}}+e^{-\lambda (t-s)}\|\nabla^2[g_{1},g_{2}](s)\|_{L^{2}}\right\}ds.
\end{aligned}
\end{eqnarray*}
It follows from $\eqref{sec5.ggg}$ that
\begin{eqnarray*}
\|[g_{1},g_{2}](s)\|_{L^{1}\cap\dot{H}^{2}}
\leq C\|U(s)\|^{2}_{3}\leq
C\epsilon_{3}^{2}(U_{0})(1+s)^{-\frac{3}{2}}.
\end{eqnarray*}
We have from Lemma \ref{halft} that
\begin{eqnarray*}
 \|[\rho,\phi](t)\|_{L^{\infty}}\leq C\epsilon_{3}(U_{0})(1+t)^{-\frac{3}{2}},
\end{eqnarray*}
which together with \eqref{rate.brho} and  $L^{2}-L^{\infty} $ interpolation gives
\begin{eqnarray}\label{rholq}
\|[\rho,\phi](t)\|_{L^{q}}\leq C
\epsilon_{3}(U_{0})(1+t)^{-\frac{3}{2}+\frac{3}{2q}}, \ \  2\leq q \leq \infty.
\end{eqnarray}

Applying the  second  estimate on $u$ with  $m=0,\ q=r=2,\ s=p=1$ in Proposition
\ref{thm.decayrup} to the mild form \eqref{sec5.U}, we get
\begin{equation*}
\|u(t)\| \leq C(1+t)^{-\frac{5}{4}}\left(
\|U_{0}\|_{L^{1}\cap L^{2}}+\|\nabla \phi_{0}\|\right)+\int_{0}^{t}(1+t-s)^{-\frac{5}{4}}\|[g_{1}(s),g_{2}(s)]\|_{L^{1}\cap L^{2}}ds.
\end{equation*}
By  \eqref{V.decay}, it follows that
\begin{equation*}
\|[g_{1}(s),g_{2}(s)]\|_{L^{1}\cap L^{2}}
\leq C\|U(t)\|^{2}_{2}\leq
C\epsilon_{3}^{2}(U_{0})(1+t)^{-\frac{3}{2}}.
\end{equation*}
Therefore, one has by using Lemma \ref{halft} that
\begin{eqnarray}\label{uL2}
\|u(t)\| \leq C\epsilon_{3}(U_{0})(1+t)^{-\frac{5}{4}}.
\end{eqnarray}

For the $L^{\infty}$-norm decay rate, applying
the  second  estimate on $u$ with  $m=0,\ q=\infty,\ r=2,\ s=p=1$ in Proposition
\ref{thm.decayrup} to the mild form \eqref{sec5.U}, one has
\begin{equation*}
\begin{aligned}
\|u(t)\|_{L^{\infty}} \leq & C (1+t)^{-2}\|\rho_{0}\|_{L^{1}}+(1+t)^{-\frac{5}{2}}
 \|U_{0}\|_{L^1}+e^{-\lambda t}\| U_{0}\|_{L^{1}}\\
& +e^{-\lambda t}\|\nabla^2 U_{0}\|_{L^{2}}+e^{-\lambda t}\|\nabla ^3[\rho,\phi_{0}]\|_{L^{2}}\\
 &+C
\int_{0}^{t}\left\{(1+t-s)^{-2} \|g_{1}(s)\|_{L^1}+ (1+t-s)^{-\frac{5}{2}}
 \|[g_{1},g_{2}](s)\|_{L^1}\right\}ds\\
 &+C
\int_{0}^{t}e^{-\lambda (t-s)}\left(\|[g_{1},g_{2}](s)\|_{L^{1}\cap \dot{H}^{2}}+\|\nabla^3g_{1}(s)]\|\right)ds.
 \end{aligned}
\end{equation*}
It is straightforward to verify that for  any $0\leq s\leq t$,
\begin{equation*}
\|[g_{1},g_{2}](s)\|_{L^1}\leq C\|U(s)\| \|\nabla U\|\leq C\epsilon_{3}(U_{0})(1+s)^{-2},
\end{equation*}
\begin{eqnarray*}
\|\nabla^2[g_{1},g_{2}](s)\|_{L^{2}}
\leq C\|\nabla U(s)\|^{2}_{2}\leq
C\epsilon_{3}^{2}(U_{0})(1+s)^{-\frac{5}{2}},
\end{eqnarray*}
and
\begin{eqnarray*}
\|\nabla^3g_{1}\|_{L^{2}}
\leq C\|\nabla U(s)\|^{2}_{3}\leq
C\epsilon_{4}^{2}(U_{0})(1+s)^{-\frac{5}{2}}.
\end{eqnarray*}
Then, we have from the above inequalities and Lemma \ref{halft} that
\begin{eqnarray*}
 \|u(t)\|_{L^{\infty}}\leq C\epsilon_{4}(U_{0})(1+t)^{-2}.
\end{eqnarray*}
By \eqref{uL2} and $L^{2}-L^{\infty} $ interpolation, we have
\begin{eqnarray}\label{uLq}
\|u(t)\|_{L^{q}}\leq C
\epsilon_{4}(U_{0})(1+t)^{-2+\frac{3}{2q}}, \ \ \ 2\leq q \leq \infty.
\end{eqnarray}

\subsection{Asymptotic decay rates to the linearized problem}
In this section, we shall prove that the solution $U=[\rho,u,\phi]$ of the nonlinear Cauchy problem
\eqref{HPV3}-\eqref{NI} can be approximated by the solution of the corresponding
linearized problem $\eqref{HPVL}$-$\eqref{NLI}$ in large time with faster decay rates.

\begin{proposition}\label{asy.small.lem}
	Suppose that $\epsilon_{4}(U_{0})>0$ is sufficiently small, and
	$U=[\rho, u,\phi]$ is the solution to the Cauchy
	problem \eqref{HPV3}-\eqref{NI} with initial data $U_0$. Let $\FP_{1}, \FP_{2}$ and $\FP_{3}$ denote the projection operators along the component $\rho, u, \phi$ of the solution $e^{tL}U_{0}$, respectively.  Then it holds that for any $t\geq 0$ and $2\leq q \leq \infty$,
	\begin{eqnarray}
	&&\left\|
	\rho(t)-\FP_{1}e^{tL}U_{0}\right\|_{L^{q}}\leq C
	(1+t)^{-2+\frac{3}{2q}}, \label{asy.small.dec.rho}\\
	&&\left\|u(t)-\FP_{2}e^{tL}U_{0}\right\|_{L^{q}}\leq
	C (1+t)^{-\frac{5}{2}+\frac{3}{2q}},\label{asy.small.dec.u}\\
	&&\left\|\phi(t)-\FP_{3}e^{tL}U_{0}\right\|_{L^{q}}\leq C
	(1+t)^{-2+\frac{3}{2q}}.\label{asy.small.dec.E}
	\end{eqnarray}
\end{proposition}

\begin{proof}
	We rewrite each component of solutions $U=[\rho, u,\phi]$ to  \eqref{HPV3}-\eqref{NI} as the mild solution forms
	by the Duhamel's principle:
	\begin{equation}\label{opt.eq1}
	\rho(t,x)=\FP_{1}e^{tL}U_{0}+\int_0^t\FP_{1} e^{(t-s)L}[\nabla
	\cdot f_{1}(s),g_{2}(s),0]ds,
	\end{equation}
	\begin{equation}\label{opt.eq2}
	u(t,x)=\FP_{2}e^{tL}U_{0}+\int_0^t\FP_{2}e^{(t-s)L}[\nabla
	\cdot f_{1}(s),g_{2}(s),0]ds
	\end{equation}
	and
	\begin{equation*}
	\phi(t,x)=\FP_3e^{tL}U_{0}+\int_0^t\FP_3e^{(t-s)L}[\nabla
	\cdot f_{1}(s),g_{2}(s),0]ds.
	\end{equation*}
	
Denote $N(s)=[\nabla \cdot f_{1}(s),g_{2}(s),0]$. In what follows we only prove
	\eqref{asy.small.dec.rho} and \eqref{asy.small.dec.u}, and the estimate \eqref{asy.small.dec.E} can be
	proved in a similar way. One can apply the linear estimate on
	$\FP_{1}e^{tL}N_{0}$  to the mild form
	\eqref{opt.eq1}  by setting $m=0,\ q=r=2,\ s=p=1$ in Proposition \ref{thm.decayrup}, so as to obtain
	\begin{eqnarray}\label{asy.lg.prf.eq1}
\begin{aligned}
	\left\|\rho(t)-\FP_{1}e^{tL}U_{0}\right\|
	&\leq \displaystyle \int_0^t\left\|\FP_{1}e^{(t-s)L}[\nabla
	\cdot f_{1}(s),g_{2}(s),0]\right\|ds\\
	&\leq  C\int_0^t(1+t-s)^{-\frac{5}{4}}\left(\|f_{1}(s)\|_{L^{1}}+\|N(s)\|_{L^{1}\cap
		L^{2}}\right)ds.
\end{aligned}
	\end{eqnarray}
	Here the divergence form of the first source term have been used.
	Recalling the definition $\eqref{sec5.ggg}$, we can verify that
	\begin{eqnarray*}
		\|N(s)\|_{L^{1}\cap L^{2}}+\|f_{1}(s)\|_{L^{1}}\leq
		C\|U(s)\|_{2}^{2}\leq C\epsilon_{3}^{2}(U_{0})(1+s)^{-\frac{3}{2}}.
	\end{eqnarray*}
	Then substituting this estimate into \eqref{asy.lg.prf.eq1}, and using Lemma \ref{halft}, we get that
	\begin{equation*}
	\left\|
	\rho(t)-\FP_{1}e^{tL}U_{0}\right\|\leq C
	(1+t)^{-\frac{5}{4}}.
	\end{equation*}
	One can apply the linear estimate on
	$\FP_{1}e^{tL}N_{0}$  to the mild form
	\eqref{opt.eq1}  by setting $m=0,\ q=\infty, r=2,\ s=p=1$ in Proposion \ref{thm.decayrup}, so as to obtain
	\begin{eqnarray}\label{asy.lg.prf.eq3}
\begin{aligned}
	\left\|\rho(t)-\FP_{1}e^{tL}U_{0}\right\|_{L^{\infty}}
	&\leq \displaystyle \int_0^t\left\|\FP_{1}e^{(t-s)L}[\nabla
	\cdot f_{1}(s),g_{2}(s),0]\right\|_{L^{\infty}}ds\\
	&\leq  C\int_0^t(1+t-s)^{-2}\left(\|f_{1}(s)\|_{L^{1}}+\|N(s)\|_{L^{1}\cap
		\dot{H}^{2}}\right)ds.
\end{aligned}
	\end{eqnarray}
	It follows from $\eqref{sec5.ggg}$ that
	\begin{eqnarray*}
		\|N(s)\|_{L^{1}\cap \dot{H}^{2}}+\|f_{1}(s)\|_{L^{1}}\leq
		C\|U(s)\|_{3}\|\nabla U(s)\|_{2}+C\|\rho(s)\|\|u(s)\|\leq C\epsilon_{3}^{2}(U_{0})(1+s)^{-2},
	\end{eqnarray*}
where the $L^2$ time-decay rate \eqref{uL2} of $u$ has been used. 	
	Then substituting this estimate into \eqref{asy.lg.prf.eq3},  and using Lemma \ref{halft}, we get that
	\begin{equation*}
	\left\|
	\rho(t)-\FP_{1}e^{tL}U_{0}\right\|_{L^{\infty}}\leq C
	(1+t)^{-2}.
	\end{equation*}
By $L^{2}-L^{\infty} $ interpolation, we have
\begin{eqnarray*}\label{diffrhoLq}
\left\|
	\rho(t)-\FP_{1}e^{tL}U_{0}\right\|_{L^{q}}\leq C
\epsilon_{N}(U_{0})(1+t)^{-2+\frac{3}{2q}}, \ \ 2\leq q \leq \infty.
\end{eqnarray*}

	Applying the linear estimate on $\FP_{2}e^{tL}N_{0}$ to the mild form \eqref{opt.eq2} by letting
	$m=0,\ q=r=2,\ s=p=1$ in Proposition \ref{thm.decayrup}
	gives
	\begin{eqnarray}\label{asy.lg.prf.eq2}
\begin{aligned}
	\left\|u (t)-\FP_{2}e^{tL}U_{0}\right\|
	&\leq \displaystyle \int_0^t\left\|\FP_{2}e^{(t-s)L}[\nabla
	\cdot f_{1}(s),g_{2}(s),0]\right\|ds\\
	&\leq  C\int_0^t(1+t-s)^{-\frac{7}{4}}\left(\|f_{1}(s)\|_{L^{1}\cap \dot{H}^{2}}+\|N(s)\|_{L^{1}\cap
		L^{2}}\right)ds.
\end{aligned}
	\end{eqnarray}
	Using the time-decay rates \eqref{rate.brho} and \eqref{uL2}, we can estimate $L^{1}$ norm and $\dot{H}^{2}$ norm of  $f_{1}$ as follows:
	\begin{eqnarray*}
		\|f_{1}(s)\|_{L^{1}}\leq \|u\| \|\rho\|\leq
		C\epsilon_{3}^{2}(U_{0})(1+s)^{-\frac{5}{4}}(1+s)^{-\frac{3}{4}}\leq
		C\epsilon_{3}^{2}(U_{0})(1+s)^{-2},
	\end{eqnarray*}
	\begin{eqnarray*}
		\|f_{1}(s)\|_{\dot{H}^{2}}\leq \|u\|_{L^{\infty}} \|\nabla^2\rho\|+\|\rho\|_{L^{\infty}} \|\nabla^2u\|\leq C\|\nabla U\|_{2}^2\leq
		C\epsilon_{3}^{2}(U_{0})(1+s)^{-\frac{5}{2}}.
	\end{eqnarray*}
	For other terms with the first-order derivative, like $\rho\nabla
	\cdot u$, one has
	\begin{eqnarray*}
		\|\rho\nabla \cdot u\|_{L^{1}\cap L^{2}}\leq \|\nabla
		u\| \|\rho\|+\|\rho\|_{L^{\infty}}\|\nabla
		u\| \leq
		C\epsilon_{3}(U_{0})(1+s)^{-\frac{5}{4}}\epsilon_{3}
		(U_{0})(1+s)^{-\frac{3}{4}}\leq C\epsilon_{3}^{2}(U_{0})(1+s)^{-2},
	\end{eqnarray*}
	and similarly it follows that
	\begin{eqnarray*}
		\|u \cdot\nabla\rho \|_{L^{1}\cap L^{2}}+\|u
		\cdot\nabla u\|_{L^{1}\cap L^{2}}+\|\rho\nabla
		\rho\|_{L^{1}\cap L^{2}}\leq C\epsilon_{3}^{2}(U_{0})(1+s)^{-2}.
	\end{eqnarray*}
	Plugging the above inequalities into \eqref{asy.lg.prf.eq2}, and using Lemma \ref{halft}, we get
	\begin{equation*}
	\left\|u(t)-\FP_{2}e^{tL}U_{0}\right\|\leq
	C (1+t)^{-\frac{7}{4}}.
	\end{equation*}

	Before estimating the $L^{\infty}$-norm, we first give another $L^{p}$-$L^{q}$ time decay estimates on $u$.  It follows from the second estimate \eqref{error2} in Lemma \ref{errorL} and expression \eqref{ubar} that
	\begin{multline*}
	\|\nabla
	^{m}u(t)\|_{L^q}\leq C(1+t)^{-\frac{3}{2}(\frac{1}{p}-\frac{1}{q})-\frac{m+1}{2}}\|\rho_{0}\|_{L^{p}}
	+C(1+t)^{-\frac{3}{2}(\frac{1}{p}-\frac{1}{q})-\frac{m+1}{2}}\|\nabla U_{0}\|_{L^{p}}+C{e^{-\lambda t}}\|\nabla^mU_{0}\|_{L^{s}}\\
	+C{e^{-\lambda t}}
\|\nabla^{m+[3(\frac{1}{r}-\frac{1}{q})]_+}U_{0}\|_{L^r}+C{e^{-\lambda t}}\|\nabla^{m+1+[3(\frac{1}{r}-\frac{1}{q})]_+}[\rho_{0},\phi_{0}]\|_{L^{r}}.
	\end{multline*}
	Applying the linear estimate on $\FP_{2}e^{tL}N_{0}$ to the mild form \eqref{opt.eq2} by setting
	$m=0,\ q=\infty,\ r=2,\ s=p=1$ to the second estimate over $\left[0,\frac{t}{2}\right]$ in Proposition \ref{thm.decayrup}, and by setting
	$m=0,\ q=\infty,\ s=r=2,\ p=1$ to the above estimate over $\left[\frac{t}{2},t\right]$ give
	\begin{equation}\label{asy.lg.prf.eq4}
	\begin{aligned}
	\left\|u (t)-\FP_{2}e^{tL}U_{0}\right\|_{L^{\infty}}
	&\leq  \displaystyle \int_0^t\left\|\FP_{2}e^{(t-s)L}[\nabla
	\cdot f_{1}(s),g_{2}(s),0]\right\|_{L^{\infty}}ds\\
	&\leq  C\int_0^{\frac{t}{2}}(1+t-s)^{-\frac{5}{2}}\left(\|f_{1}(s)\|_{L^{1}\cap \dot{H}^4}+\|N(s)\|_{L^{1}\cap
	\dot{H}^{2}}\right)ds\\
		&+C\int_{\frac{t}{2}}^t (1+t-s)^{-2}\left(\|\nabla \cdot f_{1}(s)\|_{L^{1}}+\|\nabla N(s)\|_{L^{1}}\right)ds\\
		&+C\int_{\frac{t}{2}}^t e^{-\lambda(t-s)}\left(\|N(s)\|_{L^{2}}+\|\nabla^2N(s)\|_{L^{2}}+\|\nabla^3\left(\nabla \cdot f_{1}(s)\right)\|_{L^{2}}\right)ds.
		\end{aligned}
	\end{equation}
	From the definition $\eqref{sec5.ggg}$ and the time-decay rates \eqref{rate.brho} and \eqref{uL2}, we have
	\begin{eqnarray*}
	\begin{aligned}
		\|f_{1}(s)\|_{L^{1}\cap \dot{H}^4}\leq &\|u\| \|\rho\|+\|U(s)\|_{3} \|\nabla U(s)\|_{3}\\[2mm]
		\leq
		&C\epsilon_{4}^{2}(U_{0})(1+s)^{-\frac{5}{4}}(1+s)^{-\frac{3}{4}}\\[2mm]
		\leq
		&C\epsilon_{4}^{2}(U_{0})(1+s)^{-2},
		\end{aligned}
	\end{eqnarray*}
	and
	\begin{eqnarray*}
		\|N(s)\|_{L^{1}\cap \dot{H}^{2}}\leq \|U(s)\|_{3} \|\nabla U(s)\|_{2}\leq
		C\epsilon_{3}^{2}(U_{0})(1+s)^{-\frac{3}{4}}(1+s)^{-\frac{5}{4}}\leq
		C\epsilon_{3}^{2}(U_{0})(1+s)^{-2}.
	\end{eqnarray*}
	For the  terms in $\nabla \cdot f_{1}$, by using the time-decay rates \eqref{sec5.highBE}, \eqref{sec5.highu1}, \eqref{rate.brho} and \eqref{uL2}, one has
	\begin{eqnarray*}
		\|u\cdot \nabla \rho\|_{L^{1}}\leq \|u\|\|\nabla \rho\|\leq
		C\epsilon_{3}(U_{0})(1+s)^{-\frac{5}{4}}\epsilon_{3}
		(U_{0})(1+s)^{-\frac{5}{4}}\leq C\epsilon_{3}^{2}(U_{0})(1+s)^{-\frac{5}{2}},
	\end{eqnarray*}
	\begin{eqnarray*}
	\|\rho\nabla \cdot u\|_{L^{1}}\leq  \|\rho\|\|\nabla \cdot u\|\leq
		C\epsilon_{3}(U_{0})(1+s)^{-\frac{3}{4}}\epsilon_{3}
		(U_{0})(1+s)^{-\frac{7}{4}}\leq C\epsilon_{3}^{2}(U_{0})(1+s)^{-\frac{5}{2}}.
	\end{eqnarray*}
	Similarly, it follows that
	\begin{eqnarray*}
		\|\nabla (\rho \nabla \rho)\|_{L^{1}}\leq \|\nabla \rho \|^2+\|\rho\|\|\nabla ^2\rho\|,
	\end{eqnarray*}
	where $\|\nabla ^2\rho \|$ is bouned by $\epsilon_{4}(U_{0})(1+s)^{-\frac{7}{4}}$ in \eqref{sec5.highrho} and hence,
	\begin{eqnarray*}
		\|\nabla (\rho \nabla \rho)\|_{L^{1}}\leq \|\nabla \rho \|^2+\|\rho\|\|\nabla ^2\rho\|\leq \epsilon_{4}^{2}(U_{0})(1+s)^{-\frac{5}{2}}.
	\end{eqnarray*}
	Moreover,
	\begin{eqnarray*}
	\|N(s)\|_{L^{2}}+\|\nabla^2N(s)\|_{L^{2}}+\|\nabla^3\left(\nabla \cdot f_{1}(s)\right)\| \leq \|\nabla U \|_{3}^2\leq \epsilon_{4}^{2}(U_{0})(1+s)^{-\frac{5}{2}}.
	\end{eqnarray*}
	Plugging the above inequalities into \eqref{asy.lg.prf.eq4} and using Lemma \ref{halft}, one has
	\begin{equation*}
	\left\|u(t)-\FP_{2}e^{tL}U_{0}\right\|_{L^{\infty}}\leq
	C (1+t)^{-\frac{5}{2}},
	\end{equation*}
which together with $L^{2}-L^{\infty} $ interpolation leads to
\begin{eqnarray*}
\left\|u(t)-\FP_{2}e^{tL}U_{0}\right\|_{L^{q}}\leq C
\epsilon_{4}(U_{0})(1+t)^{-\frac{5}{2}+\frac{3}{2q}}
\end{eqnarray*}
for $2\leq q \leq \infty$.
	This completes the proof of Proposition \ref{asy.small.lem}.
\end{proof}

\subsection{Proof of Theorem \ref{thm.GE}}
With all necessary {\it a priori} estimates derived in preceding sections, we are ready to prove Theorem \ref{thm.GE}. By Theorem \ref{pro.2.1}, \eqref{rholq} and \eqref{uLq}, we get the global existence of solutions to \eqref{HPV1}, \eqref{1.2} with  \eqref{gl1}-\eqref{LQ} in Theorem \ref{thm.GE}. It remains only to show \eqref{nu.decay3}.

For the solution $U_2(x,t)=[\rho_2,u,\phi_2]$ of the
Cauchy problem \eqref{HPV3}-\eqref{NI} and the desired large-time asymptotic profile
$\tilde{U}(x,t)=[\tilde{\rho},\tilde{u},\tilde{\phi}]$, their difference
can be rewritten as
\begin{equation*}
U_2-\tilde{U}=(U_2-e^{tL}U_0) +(e^{tL}U_0-\tilde{U}),
\end{equation*}
that is,
\begin{equation*}
\begin{aligned}
\rho-\bar{\rho}-\tilde{\rho}=\rho_2-\tilde{\rho}&=\left(\rho_2-\FP_{1}e^{
	tL}U_0\right)+\left(\FP_{1}e^{ tL}U_{0}-\tilde{\rho}\right),\\
u-\tilde{u}&=\left(u-\FP_{2}e^{
	tL}U_0\right)+\left(\FP_{2}e^{tL}U_{0}-\tilde{u}\right),\\
\phi-\bar{\phi}-\tilde{\phi}=\phi_2-\tilde{\phi}&=\left(\phi_2-\FP_3e^{
	tL}U_0\right)+\left(\FP_3e^{tL}U_{0}-\tilde{\phi}\right).
\end{aligned}
\end{equation*}
Recall that the solution $[\rho_1,u,\phi_1]$ of the linearized Cauchy problem \eqref{HPVL}-\eqref{NLI} can be written as  $[\rho_1, u, \phi_1]= [\FP_{1}, \FP_{2}, \FP_{3}]e^{ tL}U_{0}$. Then \eqref{nu.decay3} follows from Proposition \ref{thm.decaypar} (replacing $[\rho,u, \phi]$ by $[\rho_1,u, \phi_1]$) and Proposition \ref{asy.small.lem} (replacing $[\rho,u, \phi]$ by $[\rho_2,u, \phi_2]$). This completes the proof of Theorem \ref{thm.GE}.

\medskip

\noindent{\bf Acknowledgments:}
Q.Q. Liu was supported by the National Natural Science
Foundation of China (No. 12071153), and the Fundamental Research Funds for the Central Universities (No. 2020ZYGXZR032). H.Y. Peng was supported from the National Natural Science Foundation of China (No. 11901115), Natural Science Foundation of Guangdong Province (No. 2019A1515010706) and Grant from GDUT (No. 220413228). Z.A. Wang was supported in part by the Hong Kong Research Grant Council  General Research Fund No. PolyU 15304720 (project id: P0032967).

\end{document}